\documentclass{article}
\usepackage{amsmath, amssymb, bm}
\usepackage{mathrsfs}
\usepackage{amsthm}
\allowdisplaybreaks[4]
\theoremstyle{plain}
\newtheorem{theorem}{Theorem}
\newtheorem{proposition}{Proposition}

\newtheorem{lemma}{Lemma}
\theoremstyle{definition}
\newtheorem{remark}{Remark}
\newtheorem{assumption}{Assumption}
\newcommand{\ord}{\mathrm{ord}}
\title{On the distribution of $\log |L(\sigma, \chi)|$  and $\log L(\sigma, \chi_D)$ in the modulus aspect}
\author{Manami Hosoi and Yumiko Umegaki}
\date{}
\begin{document}
\maketitle
%
%
\begin{abstract}
\par
Let $\sigma$ be a real number with $\sigma>1/2$.
For the certain average of values $\log |L(\sigma, \chi)|$ in the prime conductor aspect, we show that it can be expressed by an integral involving the same density function as the one which constructed for  the certain average of values of difference between logarithms of two symmetric power $L$-functions in the level aspect (see \cite{matsumoto-umegaki-ijnt}). 
For the distribution of values $\log L(\sigma, \chi_D)$ and $L'/L(\sigma, \chi_D)$ in the $D$-aspect, where $\chi_D$ is a real character attached to a fundamental discriminant $D$, we also show that there exists a density function.
\end{abstract}
%
%
\section{Introduction and main results}
\par
Let $\mathcal{L}(s, \chi)$ be either $\log L(s, \chi)$ or $L'/L (s,\chi)$, where $L(s, \chi)$ is a Dirichlet $L$-function associated to primitive Dirichlet characters of prime conductor $q$.
In 2011, Ihara and Matsumoto showed that the averages of values $\Phi(\mathcal{L}(s, \chi))$ at a fixed point $s=\sigma+it$ ($\sigma>1/2$) can be expressed as an integral involving a density function which called ``\textit{$M$-function}'' (see \cite{ihara-matsumoto-QJM} and \cite{ihara-matsumoto-moscow}), where $\Phi$ is any bounded continuous function or any compactly supported Riemann integrable function (see \cite{matsumoto}). 
Under Generalized Riemann Hypothesis (GRH), Ihara and Matsumoto \cite{ihara-matsumoto-moscow} constructed an $M$-function $\mathcal{M}_{\sigma}(w)$ which satisfies
\[
\lim_{\substack{q\to\infty\\ q\;:\;\text{prime}}} 
\frac{1}{q-2}
\underset{\chi\in X(q)}{\sum} 
\Phi(\mathcal{L}(s,\chi)) =\int_{\mathbb{C}} \mathcal{M}_{\sigma}(w)\Phi(w) |dw|,
\]
where $q>1$ is a prime number and $X(q)$ is a set of all primitive Dirichlet characters of conductor $q$.
On the other hand, Ihara and Matsumoto \cite{ihara-matsumoto-QJM} also obtained 
\[
\lim_{m\to\infty}
\frac{1}{\pi(m)}
\sum_{\substack{q\leq m\\ q\;:\;\text{prime}}} 
\frac{1}{q-2}
\underset{\chi\in X'(q, s)}{{\sum}} 
\Phi(\mathcal{L}(s,\chi)) =\int_{\mathbb{C}} \mathcal{M}_{\sigma}(w)\Phi(w) |dw|,
\]
unconditionally (not assuming GRH), where 
$\pi(m)$ means the number of the primes less than or equal to a given number $m$, 
\begin{align*}
X'(q, s)
=&
\{\chi\in X(q)\mid s\in G_{\chi}\},
\\
G_{\chi}
=&
\{z\in\mathbb{C} \mid \Re(z)>1/2\}\setminus\bigcup_{\rho} B_{\rho}(\chi),
\\
B_{\rho}(\chi)
=&
\{z\in\mathbb{C} \mid 1/2<\Re(z)\leq \Re(\rho), \; \Im(z)=\Im(\rho)\}
\end{align*}
and $\rho$ is a possible zero of $L(z,\chi)$ with $\Re(\rho)>1/2$.
We define the value of $\log L(z, \chi)$ for $z=u+iv \in G_{\chi}$ with $1/2< u \leq 1$ by the analytic continuation along the horizontal path $z=u'+iv$ for $u' \geq u$.
One of advantages of Ihara and Matsumoto's argument is that we can fined an $M$-function by only considering the case that $\Phi(w)=\psi_{z_1, z_2}(w)=\exp(i(z_1\overline{w}+z_2w)/2)$ for $|z_1|, |z_2|<R$, where $R$ is any positive real number. 
\par
Ihara and Matsumoto discussed the $M$-functions for the averages of values $\Phi(\log L(s, \chi))$ by two ways (see \cite{ihara-matsumoto-QJM} and \cite{ihara-matsumoto-moscow}). 
In \cite{ihara-matsumoto-QJM}, for a fixed point $s=\sigma+it$, 
they showed that we can construct the $M$-function for the average of $\psi_{z_1, z_2}(\log L(s, \chi))$ 
by a function 
\[
g_{\sigma, p}(t_p)=-\log (1-t_p p^{-\sigma})
\]
at $t_p= \chi(p)p^{-it}$.
Their argument in \cite{ihara-matsumoto-QJM} was applied to the study of the distribution of values of logarithms of automorphic $L$-functions in the level aspect.
In 2018, Matsumoto and the second author~\cite{matsumoto-umegaki-ijnt} studied the distribution of values of difference between logarithms of two symmetric power $L$-functions in the level aspect at a fixed point $s=\sigma>1/2$ by $g_{\sigma, p}(t_p)+g_{\sigma, p}(t_p^{-1})$.
They obtained an $M$-function for the certain averages of them by the strategy in \cite{ihara-matsumoto-QJM}.  
After that, Lebacque, Matsumoto, Mine and the second author \cite{LMMU} studied the distribution of values of logarithms of symmetric power $L$ functions and obtained $M$-functions for the certain averages of them in the level aspect by the strategy in \cite{ihara-matsumoto-QJM}, too.
Moreover we can see that the $p$-part of the $M$-functions is related to Sato-Tate measure.
On these results, the strategy in \cite{ihara-matsumoto-QJM} is may better than one in \cite{ihara-matsumoto-moscow} for focusing directly on the relationship between the properties of the Fourier coefficients of cusp forms and the constructing $p$-part of the $M$-functions, because we can see the involvement of Sato-Tate measure.
\par 
On the other hand, by the results in Ihara and Matsumoto~\cite{ihara-matsumoto-moscow}, Mourtada and Murty~\cite{mourtada-murty} obtained an $M$-function for the average of values $L'/L(\sigma, \chi_D)$ in the $D$-aspect, where $\chi_D$ is a real character attached to a fundamental discriminant $D$.
Then we can say that the both of methods in \cite{ihara-matsumoto-moscow} and \cite{ihara-matsumoto-QJM} are important.
We compare these methods in Remark~\ref{56} below.
\par 
The first aim of this paper is to prove Theorem~\ref{hosoi1} by the strategy in \cite{ihara-matsumoto-QJM}. 
By the proof of Theorem~\ref{hosoi1}, we can see the $M$-function in Theorem~\ref{hosoi1} is same as the one in \cite{matsumoto-umegaki-ijnt}.
%
%
\begin{theorem}\label{hosoi1}
For fixed $\sigma>1/2$, 
there exists a continuous non-negative function $\mathcal{M}_{\sigma}(x)$ such that
\[
\lim_{\substack{q\to\infty\\ q\;\text{: prime}}}
\frac{1}{|X(q)|}
\sum_{\chi\in X'(q, \sigma)}
\Phi(2\log |L(\sigma, \chi)|) 
=
\int_{\mathbb{R}} 
\mathcal{M}_{\sigma}(u)\Phi(u) 
\frac{du}{\sqrt{2\pi}}.
\] 
\end{theorem}
\begin{remark}
Ihara and Matsumoto \cite{ihara-matsumoto-moscow} discussed this type of average under GRH for $\mathcal{L} (s, \chi)$.
\end{remark}
%
%
%
%
%
\par
The second aim of this paper is to prove Theorem~\ref{main1} below.
Under GRH, Mourtada and Murty~\cite{mourtada-murty}  constructed a density function $\mathcal{Q}_{\sigma}(x)$, 
which
\[
\lim_{Y\to\infty} 
\frac{1}{N(Y)}
\underset{|D|\leq Y}{{\sum}^{\ast}}
\Phi \left(\frac{L'}{L}(\sigma, \chi_D)\right)
=
\int_{\mathbb{R}} \mathcal{Q}_{\sigma}(u) \Phi(u) 
\frac{du}{\sqrt{2\pi}}
\]
holds, where $\chi_D$ is a real character attached to $D$, 
${\sum}^{\ast}$ is the sum over fundamental discriminants
and $N(Y)$ means $\sum_{|D|\leq Y}^{\ast}1$.
Mourtada and Murty calculated the left-hand side of the above equation in the case $\Phi(u)=\psi_a(u)$ for $|a|<R$, where $\psi_a(u)=\exp(iau)$, and they constructed $\mathcal{Q}_{\sigma}(u)$. 
Mourtada and Murty's argument is supported by the results of Ihara and Matsumoto~\cite{ihara-matsumoto-moscow}. 
In this paper, 
we prove the following theorem unconditionally.
%
%
\begin{theorem}\label{main1}
For fixed $\sigma>1/2$, 
there exists a density function $\mathcal{Q}_{\sigma}(u)$ 
such that
 \[
\lim_{Y\to\infty} \frac{1}{N(Y)}
\underset{|D|\leq Y}{{\sum}^{\dag}} 
\Phi \left(\mathcal{L}(\sigma, \chi_D)\right)
=
\int_{\mathbb{R}} 
\mathcal{Q}_{\sigma}(u) \Phi(u) 
\frac{du}{\sqrt{2\pi}},
\]
where ${\sum}^{\dag}$ is the sum over fundamental discriminant $D$ which $L(\sigma', \chi_D)\neq 0$ for $\sigma\leq \sigma'$.
\end{theorem}
\begin{remark}
In the proofs of Theorem~\ref{hosoi1} and \ref{main1}, we use Lemma~\ref{arg} and the upper bound \eqref{for_th1} below.
These are obtained by the argument of Lemma~3.6 in Akbary and Hamieh \cite{akbary-hamieh}. 
For the case $\mathcal{L}(\sigma, \chi_D)=L'/L(\sigma, \chi_D)$, Theorem~\ref{main1} menas that it is possible to remove the GRH assumption in Mourtada and Murty's result, as Akbary and Hamieh mentioned in \cite{akbary-hamieh}.
\end{remark} 
%
%
\begin{remark}\label{56}
In this remark, we assume GRH for simplicity.
 Ihara and Matsumoto constructed the density functions for the sum of values $\log L(s,\chi)$ over Dirichlet characters by two ways. One is \cite{ihara-matsumoto-QJM} and the other is \cite{ihara-matsumoto-moscow}.
 For the first aim of this paper, which is Theorem~\ref{hosoi1}, we use the similar argument as one in \cite{ihara-matsumoto-QJM}, 
because the $\mathcal{M}$-function derived from $g_{\sigma, p}(t)+g_{\sigma, p}(t^{-1})$ had been constructed in \cite{matsumoto-umegaki-ijnt} by the method in \cite{ihara-matsumoto-QJM}.
If we use the argument in \cite{ihara-matsumoto-moscow} for the proof of Theorem~\ref{hosoi1}, we need to discuss the Euler product of 
\[
\sum_{\chi\in X(q)}
\psi_a(2\log |L(\sigma, \chi)|)
 =
 \sum_{\chi\in X(q)}
 \exp(ia \log L(\sigma, \chi))
 \exp(ia \log L(\sigma, \overline{\chi}))
\]
and show that many prime factors have good estimate (see Section~\ref{section_proof_main1}).
\par
For the second aim of this paper, which is Theorem~\ref{main1}, we use the argument in \cite{ihara-matsumoto-moscow}. The reason of this is as following.
Ihara and Matsumoto~\cite{ihara-matsumoto-QJM} construct 
the partial $M$-function $\mathcal{M}_{\sigma, P}$ which satisfies
\begin{equation}\label{M=integral}
\int_{\mathbb{C}} \mathcal{M}_{\sigma, P}(w)\Phi(w) |dw|
=\int_{T_P}\Phi(g_{\sigma, P}(\bm{t}_P))d^*\bm{t}_P,
\end{equation}
where $P=P(y)$ is the finite set of prime numbers which are less than or equal to $y>1$, $\bm{t}_P=(t_p) \in T_P=\prod_{p\in P} T_p$ with $T_p=\{t_p \mid |t_p|=1\}$ and $d^* \bm{t}_P$ is the normalized Haar measute on $T_P$. 
Since the construction of $\mathcal{M}_{\sigma, P}$ relates to the integration by substitution for $M_{\sigma, \{p\}}$ and the convolutions of them, so $\mathcal{M}_{\sigma, P}$ does not depend on what kind of the average of Dirichlet characters do we study.
In the case $\Phi=\psi_{z}$ (this case is essential), the idea of Ihara matsumoto \cite{ihara-matsumoto-QJM} is to prove
the existence of the function $\mathcal{M}_{\sigma}(u)$ such that 
\begin{equation}\label{avg=M}
\lim_{\substack{q\to\infty\\ q \text{: prime}}}
\frac{1}{q-2}
\sum_{\chi\in X(q)}
 \Phi(\log L(s, \chi))
=
\int_{\mathbb{C}} \mathcal{M}_{\sigma}(w)\Phi(w) |dw|
\end{equation}
by showing that the left-hand side of \eqref{avg=M} and the right hand side of \eqref{M=integral} as $y\to\infty$ are the same.
Since $\mathcal{M}_{\sigma, P}$ in \eqref{M=integral} is constructed independently of the average, we may say that the average which is the left-hand side of \eqref{avg=M} affects what kind of measure which is in the right-hand side of \eqref{M=integral} is suitable here.
On the second aim, our target is the average of values $\log L(\sigma, \chi_D)$ over real characters attached to fundamental discriminants $D$ with $|D|\leq Y$.
Even the simplest case which $\Psi(u)=u$ and $s=\sigma>1$, our target is not same as the left hand side of \eqref{avg=M}.
In fact, we can see the main terms of
\[
\frac{1}{q-1}
\sum_{\chi\in X(q)}
 \log L(\sigma, \chi)
\quad
\mathrm{and}
\quad
\frac{1}{N(Y)}
\underset{|D|\leq Y}{{\sum}^{\ast}} 
\log L(\sigma, \chi_D)
\]
are different, because the main term of the former summation comes form the $n$th Dirichlet coefficients of $\log L(s,\chi) $ with $n\equiv 1\pmod{q}$ and the main term of the latter summation comes form the $n$th Dirichlet coefficients of $\log L(\sigma, \chi_D)$ which $n$ is square.
Then naturally we think that the density function for the distribution  of values $\log L(\sigma, \chi_D)$ is not similar to $\mathcal{M}_{\sigma}$ in \eqref{avg=M}.
Therefore we give the proof of Theorem~\ref{main1} by the method in \cite{ihara-matsumoto-moscow}.
\par
We may be able to discuss $M$-functions by the method in \cite{ihara-matsumoto-QJM} if we find a suitable measure. 
For example we can construct an $M$-function of the value-distribution of symmetric power $L$-functions by the method in \cite{ihara-matsumoto-QJM} and the Sato-Tate measure (see\cite{LMMU}).
\end{remark} 
%
%
%
\subsection*{Acknowledgement}
The authors would like to thank to professor Kohji Matsumoto for his comments.
The authors would also like to our deep appreciation to professor Masahiro Mine for his comments and suggestions.
%
%
%
%
\section{Preparation}
At first, we prepare the following assumption. 
\begin{assumption}\label{ass}
For $\varepsilon_0>0$, let $Y$ be a real number with $Y\geq\max\{e^{1/\varepsilon_0}, e^2\}$.
Let $\delta$ be a positive number less than $1$.
We assume that $L(s,\chi)\neq 0$
in the domain $\mathscr{D}$ which is a rectangle defined by
$1/2+\delta/16\leq \Re(s) \leq 1$ and $|\Im(s)|\leq 2\log Y$. 
\end{assumption}
Let $X_{\mathscr{D}}(q)$ be the set of primitive Dirichlet characters of conductor $q$ which Assumption~\ref{ass} holds.
In the proofs of Theorem~\ref{hosoi1} and \ref{main1}, we use the estimates of $|\psi_x(\mathcal{L}(s,\chi))|$ and $|\mathcal{L}(s,\chi)|$ in Lemma~\ref{arg} and \eqref{for_th1} for $\chi\in X_{\mathscr{D}}(q)$, and we use \eqref{montgomery} or \eqref{jutila} for $\chi\notin X_{\mathscr{D}}(q)$.
Let $N(\sigma, T, \chi)$ be the number of zeros $\rho$ of $L(s, \chi)$ with $\Re(\rho)\geq \sigma$ and $|\Im s|\leq T$.
We know the zero-density estimate 
\begin{equation}\label{montgomery}
\sum_{\chi\in X(q)} N(\sigma, T, \chi) 
\ll
(qT)^{A(\sigma)}(\log qT)^{14},
\end{equation}
where $A(\sigma)<1$ for $1/2<\sigma$ (see Theorem~12.1 in Montogomery~\cite{montogomery_book}). 
By Jutila~\cite{jutila}, we know
\begin{equation}\label{jutila}
\sum_{\chi\in S(Y)} N(\sigma, T, \chi)
\ll_{\varepsilon} (YT)^{(7-6\sigma)/(6-4\sigma)+\varepsilon},
\end{equation}
where $S(Y)$ is the set of all real primitive characters of conductor at most $Y$.
By these estimates, we can see an upper bound of the number of characters which Assumption~\ref{ass} does not hold.
\par
In this section, we prove 
\begin{lemma}\label{arg}
Suppose $0<\delta<1$ and $Y\geq\max\{e^{1/\varepsilon_0}, e^2\}$ for any positive constant $\varepsilon_0>0$. 
Let $\chi$ be a primitive Dirichlet character of conductor $q\leq Y$ (here, $q$ does not need to be a prime number).
Under Assumption~\ref{ass}, for $s=\sigma+it$ with $(1+\delta)/2\leq \sigma \leq 1$ and $|t|\leq \log Y$, there exist positive constants $c(\delta)$ and $a(\delta)$  $(0< a(\delta) <1 )$ depending on $\delta$, such that
\[
|\psi_x(\mathcal{L}(\sigma +it , \chi))|
\leq
\exp\big(R c(\delta)\varepsilon_0^{(1-a(\delta))}\log Y\big),
\]
where $|x|<R$ for any $R>0$.
\end{lemma}
\begin{proof}
The argument in this proof is similar to the proof of Lemma~3.6 in Akbary and Hamieh's work \cite{akbary-hamieh}.
Let $s_0=2+it$, $r_0=3/2-\delta/16$ and 
\[
\mathscr{D}_0=\bigcup_{|t|\leq \log Y}\{s\in\mathbb{C}\mid |s-s_0|\leq r_0\}
\]  
which is included in $\mathscr{D}$.
In $\mathscr{D}_0$, we put
\[
\mathcal{H}(z)=\log L(z+s_0, \chi)-\log L(s_0, \chi)
\]
and consider the circles with radii $r_j$ $(0\leq j \leq 3)$ whose center is $s_0$, where $r_1=3/2-\delta/8$, $r_2=3/2-\delta/4$ and $r_3=1-\delta$. Here we know $1/2<r_3<2-\sigma<r_2<r_1<r_0$.
For $|z|\leq r_0$, we know there exists an absolute constant $C$ such that
\begin{equation}\label{trivial_logL}
\Re(\mathcal{H}(z))=\log |L(z+s_0)| -\log |L(s_0, \chi)|
\leq C\log q(|t|+1)
\end{equation}
by the trivial bound of the Dirichlet $L$-function which is
\[
L(s, \chi) \ll (q(1+|t|))^{(1+\varepsilon-\sigma)/2}
\]
for $-\varepsilon \leq \sigma \leq 1+\varepsilon$, where $0<\varepsilon <1/2$.
By \eqref{trivial_logL} and the Borel Carath{\' e}odory theorem (see Lemma~6.2 in \cite{montgomery-varghan}), we have
\[
|\mathcal{H}(z)| 
\leq \frac{2r_1 C\log q(|t|+1)}{r_0-r_1}
\ll \frac{\log q(|t|+1)}{\delta}
\]
and
\[
|\mathcal{H}'(z)| 
\leq \frac{2r_1 C\log q(|t|+1)}{(r_0-r_1)^2}
\ll \frac{\log q(|t|+1)}{\delta^2}
\]
for $|z|\leq r_1<r_0$. 
For $|z|=r_3$, we see $|\mathcal{H}(z)| \ll_{\delta} 1$ and $|\mathcal{H}'(z)|\ll_{\delta} 1$.
Let $\mathscr{H}(z)=\mathcal{H}(z)$ or $\mathcal{H}'(z)$. 
By Hadamard's three-circle theorem, 
we have
\[
\max_{|z|=2-\sigma}|\mathscr{H}(z)|
\leq
\bigg(\max_{|z|=r_3} |\mathscr{H}(z)|\bigg)^{\theta}
\bigg(\max_{|z|=r_1} |\mathscr{H}(z)|\bigg)^{1-\theta},
\]
where
\[
\theta=\frac{\log r_1 -\log(2-\sigma)}{\log r_1-\log r_3}.
\]
The upper bounds of $\mathscr{H}(z)$ which mentioned above yields
\[
\max_{|z|=2-\sigma}|\mathscr{H}(z)|
\ll_{\delta} (\log q(|t|+1))^{1-\theta}
\ll_{\delta} (\log Y)^{1-\theta}.
\]
Since
\[
1-\theta
=
\frac{\log(2-\sigma) -\log r_3}{\log r_1-\log r_3}
<
\frac{\log r_2 -\log r_3}{\log r_1-\log r_3},
\]
we put $a(\delta)=(\log r_2-\log r_3)/(\log r_1 -\log r_3)$ and we see $0<a(\delta)<1$.
Considering the case $z=-2+\sigma$ yields that there exists $c(\delta)>0$ such that
\[
|\mathcal{L}(s,\chi)|
\leq c(\delta)(\log Y)^{a(\delta)}.
\]
Since $Y>e^{1/\varepsilon_0}$,  
we see
\begin{equation}\label{for_th1}
|\mathcal{L}(s,\chi)|
\leq 
c(\delta)(\log Y)^{a(\delta)}
<
c(\delta) \varepsilon_0^{1-a(\delta)} \log Y.
\end{equation}
Hence we obtain
\begin{align*}
|\psi_x(\mathcal{L}(s, \chi))|
=&
\exp(-x\arg \mathcal{L}(s,\chi))
\leq
\exp(|x \mathcal{L}(s,\chi)|)
\\
<&
\exp(R c(\delta) \varepsilon_0^{1-a(\delta)} \log Y).
\end{align*}
\end{proof}
%
%
%
%
\section{The preparation of the proof of Theorem~\ref{hosoi1}}
\par
Let $q$ be a prime number and $P_q$ a finite set of prime numbers except for $q$. 
We denote by $P_q(y)$ a finite set of prime numbers which are less than or equal to $y$ except for $q$.
For fixed $\sigma>1/2$, 
we consider the function 
\[
g_{\sigma, p}(t_p)=-\log (1-t_p p^{-\sigma})
\]
on $T=T_p=\{t_p\in\mathbb{C} : |t_p|=1\}$ and we define 
\[
\mathcal{G}_{\sigma, P_q}(\bm{t}_{P_q})=\sum_{p\in P_q} g_{\sigma, p} (t_p)
\]
of $\bm{t}_P=(t_p)_{p\in P_q}$ on $T_{P_q}=\prod_{p\in P_q} T$.
Let $\chi$ be a Dirichlet character of modulus $q$ and put $\bm{\chi}_{P_q(y)}=(\chi(p))_{p\in P_q(y)}$.
For $\sigma>1$, we know 
\[
\lim_{y\to\infty} \mathcal{G}_{\sigma, P_q(y)}(\bm{\chi}_{P_q(y)})
=\log L(\sigma,\chi).
\]
\begin{proposition}[Proposition 3.1, \cite{matsumoto-umegaki-ijnt}]\label{matsumoto-umegaki-ijnt_prop1}
For any $\sigma>0$, there exists a non-negative function $M_{\sigma, P_q}$ defined on $\mathbb{R}$ which satisfies following two properties.
\begin{itemize}
\item The support of $M_{\sigma, P_q}$ is compact.
\item For any continuous function $\Psi$ on $\mathbb{R}$, we have
\[
\int_{T_{P_q}}\Psi(2\Re(\mathcal{G}_{\sigma, P_q}(\bm{t}_{P_q}))) d^{\ast}\bm{t}_{P_q}
=
\int_{\mathbb{R}} M_{\sigma, P_q}(u)\Psi(u)\frac{du}{\sqrt{2\pi}},
\]
where $d^{\ast}\bm{t}_{P_q}$ is the normalized Haar measure of $T_{P_q}$.
In particular, taking $\Psi\equiv 1$, we have
\[
\int_{\mathbb{R}} M_{\sigma, P_q}(u)\frac{du}{\sqrt{2\pi}}=1.
\]
\end{itemize}
\end{proposition}
Matsumoto and the second author~\cite{matsumoto-umegaki-ijnt} studied the properties of the Fourier transform of $M_{\sigma, P_q(y)}(u)$ which is defined by
\[
\widetilde{\mathcal{M}}_{\sigma, P_q(y)}(x)
=
\int_{\mathbb{R}} M_{\sigma, P_q(y)}(u)\psi_x(u)\frac{du}{\sqrt{2\pi}}.
\]
Since they showed $M_{\sigma, P_q(y)}$ and $\widetilde{\mathcal{M}}_{\sigma, P_q(y)}$ are in $L^1$, we see
\[
M_{\sigma, P_q(y)}(u)
=\int_{\mathbb{R}} \widetilde{\mathcal{M}}_{\sigma, P_q(y)}(x)\psi_{-u}(x) \frac{dx}{\sqrt{2\pi}}, 
\]
almost everywhere.
We define $\mathcal{M}_{\sigma, P_q(y)}(u)$ by the right-hand side of the above equation. 
We can see $\widetilde{\mathcal{M}}_{\sigma, P_q(y)}$ is in $L^t$ $(1\leq t \leq \infty)$ by \cite{matsumoto-umegaki-ijnt}.
Therefore we know $\mathcal{M}_{\sigma, P_q(y)}(u)$ is a continuous function for which the Fourier inversion formula holds, and $M_{\sigma, P_q(y)}=\mathcal{M}_{\sigma, P_q(y)}$ almost everywhere.
The existence of 
\[
\widetilde{\mathcal{M}}_{\sigma}(x)
=
\lim_{y\to\infty}\widetilde{\mathcal{M}}_{\sigma, P_q(y)}(x)
\]
is also proved in \cite{matsumoto-umegaki-ijnt}.
We define
\[
\mathcal{M}_{\sigma}(u)
=
\int_{\mathbb{R}}\widetilde{\mathcal{M}}_{\sigma}(x)\psi_{-u}(x)\frac{dx}{\sqrt{2\pi}}.
\]
\begin{remark}
  Matsumoto and the second author did not distinguish between  $M_{\sigma, P_q(y)}$ and $\mathcal{M}_{\sigma, P_q(y)}$ in \cite{matsumoto-umegaki-ijnt}.
They considered  $M_{\sigma, P_q(y)}$ as $\mathcal{M}_{\sigma, P_q(y)}$ by Fourier inversion formula.
However, we only know $M_{\sigma, P_q(y)}(u)=\mathcal{M}_{\sigma, P_q(y)}(u)$ almost everywhere.
But we can obtain the following Proposition~\ref{matsumoto-umegaki-ijnt_prop3}, because $\mathcal{M}_{\sigma, P_q(y)}$ is continuous.
\end{remark}
We have
\begin{proposition}[Proposition 3.3, \cite{matsumoto-umegaki-ijnt}]\label{matsumoto-umegaki-ijnt_prop3}
For $\sigma>1/2$, we have
\begin{itemize}
\item $\displaystyle \lim_{y\to\infty}\mathcal{M}_{\sigma, P_q(y)}(u)=\mathcal{M}_{\sigma}(u)$. The convergence is uniform in $u$.
\item The function $\mathcal{M}_{\sigma}(u)$ is continuous. And $\mathcal{M}_{\sigma} (u)$ is non-negative.
\item $\displaystyle \lim_{u\to\infty}\mathcal{M}_{\sigma}(u)=0$.
\item $\mathcal{M}_{\sigma}(u)$ and $\widetilde{\mathcal{M}}_{\sigma}(x)$ are Fourier duals of each other.
\item $\displaystyle\int_{\mathbb{R}}\mathcal{M}_{\sigma}(u) \frac{du}{\sqrt{2\pi}}=1$.
\end{itemize}
\end{proposition}
%
%
%
%
\par
For the proof of Theorem~\ref{hosoi1},
we show the following lemma.
\begin{lemma}\label{key_lemma_hosoi}
For $\sigma>1$, let $y>2$ be a real number which is not depend on $q$.  For $1\geq \sigma > 1/2$, let $y=\sqrt{\log\log q}$ for large $q$. Then we have
\begin{align*}
&
\lim_{\substack{q\to\infty\\ q\text{: prime}}}
\bigg(
\frac{1}{|X(q)|}
\sum_{\chi \in X(q)}
\psi_x(2\Re(\mathcal{G}_{\sigma, P_q(y)}({\bm{\chi}}_{P_q(y)})))
\\
&\hspace{1.5cm} -
\int_{T_{P_q(y)}}
\psi_x(2\Re(\mathcal{G}_{\sigma, P_q(y)}(\bm{t}_{P_q(y)}))) 
d^{\ast}\bm{t}_{P_q(y)}
\bigg)=0
\end{align*}
in $|x|\leq R$ for any $R>0$ and 
the above convergences are uniform. 
\end{lemma}
%
%
%
%
\begin{remark}
We consider the values $\log L(\sigma, \chi)$ by $\mathcal{G}_{\sigma, P_q(y)}(\bm{\chi}_{P_q(y)})$. 
In the case $\sigma>1$, we know 
$\mathcal{G}_{\sigma, P_q(y)}(\bm{\chi}_{P_q(y)}) \to \log L(\sigma, \chi)$ as $y\to\infty$.
In the case $1\geq \sigma > 1/2$, we use the estimate \eqref{sigma<1} below.
This is the reason why $y$ depends on $q$ for $1\geq \sigma >1/2$ in Lemma~\ref{key_lemma_hosoi}. 
\end{remark}
\begin{proof}
%
%
For any real number $\varepsilon$ with $0<\varepsilon<1$, we put positive real numbers
\[
c(R)=\exp\bigg(\frac{R}{\sqrt{1.4}-1}\bigg)
\quad\text{and}\quad
c_1(\varepsilon, R)
=
\bigg(\frac{576 c(R)^8}{(1-\sqrt{1.4/2})^2 \varepsilon^2}\bigg)^2
\]
which are larger than $1$.
The large prime number $q_0$ is defined as follows:
\begin{itemize}
\item In the case $1<\sigma$, we choose a prime number $q_0$ satisfying 
\[
q_0> c_1(\varepsilon, R)^{y^2}
\quad\text{and}\quad
|X(q_0)|=q_0-2> \frac{6(c(R)^2)^{y^2}}{\varepsilon}.
\]
\item In the case $1/2<\sigma\leq 1$, 
since the functions 
\[
\begin{cases}
f_1(x)=&\log x -1 
\\
f_2(x)=&e^x-x^{\log c_1(\varepsilon, R)}
\\ 
f_3(x)=&e^x-\dfrac{6}{\varepsilon}x^{\log (c(R)^2)}-2
\end{cases}
\]
are monotonically increasing for sufficiently large $x$, 
then we can choose a sufficiently large prime number $q_0$ 
which satisfies 
$f_i(\log q_0)>0$ and $f_i(x)$ are monotonically increasing for $x\geq \log q_0$ $(i=1,2,3)$. 
Any prime number $q>q_0$ satisfies $\log\log q >1$,
\[ 
q>(\log q)^{\log c_1(\varepsilon, R)}=c_1(\varepsilon, R)^{y^2}
\]
and
\[
|X(q)|=q-2 >
\frac{6(\log q)^{\log (c(R)^2)}}{\varepsilon}
=
\frac{6(c(R)^2)^{y^2}}{\varepsilon},
\]
where $y^2=\log\log q$.
\end{itemize}
So we can say that any prime number $q>q_0$ satisfies 
\begin{equation}\label{q>q_0}
q >c_1(\varepsilon, R)^{y^2} 
\quad\text{and}\quad 
|X(q)| > \frac{6(c(R)^2)^{y^2}}{\varepsilon}
\end{equation}
in both cases.
For $q>q_0$, we will show 
\begin{align*}
&
\bigg|
\frac{1}{|X(q)|}
\sum_{\chi \in X(q)}
\psi_x(2\Re(\mathcal{G}_{\sigma, P_q(y)}({\bm{\chi}}_{P_q(y)})))
\\
&\hspace{1.5cm} -
\int_{T_{P_q(y)}}
\psi_x(2\Re(\mathcal{G}_{\sigma, P_q(y)}(\bm{t}_{P_q(y)}))) d^{\ast}\bm{t}_{P_q(y)}
\bigg|
<\varepsilon.
\end{align*}
%
%
\par
Firstly, we consider
\begin{align}\label{target0}
&
\psi_x\big(2\Re\big(\mathcal{G}_{\sigma, P_q(y)}(\bm{t}_{P_q(y)})\big)\big)
=
\psi_x\big(-2\sum_{p\in P_q(y)} \log|1-t_p p^{-\sigma}|\big)
\nonumber\\
=&
\psi_x\big(\sum_{p\in P_q(y)}\big(-\log (1-t_p p^{-\sigma})-\log (1-\overline{t_p} p^{-\sigma})\big)\big)
\nonumber\\
=&
\prod_{p\in P_q(y)}
\psi_x \big( g_{\sigma, p}(t_p)\big)
\psi_x \big( g_{\sigma, p}(t_p^{-1})\big).
\end{align}
We define the polynomials $H_r(x)$ of $x$ as
\[
H_r(x)=\sum_{k=1}^r \frac{1}{k!}\delta_k(r) x^k,\quad
\delta_k(r)=\sum_{\substack{r=r_1+\ldots+r_k\\ r_1, \ldots,  r_k\geq 1}} \frac{1}{r_1\cdots r_k}
\]
and $H_0(x)=1$.
These are the coefficients of
\[
\exp (-x\log (1-t))=(1-t)^{-x}=\sum_{r=0}^{\infty} H_r(x) t^r
\]
for $|t|<1$
(see (1.2.5) in Ihara and Matsumoto~\cite{ihara-matsumoto-moscow}).
Here we mention that $H_r(x)$ is written by $G_r^*(x)$ in Ihara and Matsumoto \cite{ihara-matsumoto-QJM}.
We know
\begin{equation}\label{target1}
\psi_x(g_{\sigma, p}(t_p))
=
\frac{1}{(1-t_p p^{-\sigma})^{ix}}
=
\sum_{r=0}^{\infty}
\frac{H_r(ix)}{p^{\sigma r}}
t_p^r.
\end{equation}
Since (65) and (77) in Ihara and Matsumoto~\cite{ihara-matsumoto-QJM}, we have
\[
|H_r(ix)| \leq H_r(|x|)
\leq \sum_{k=1}^r \frac{1}{k!}
\bigg(\begin{matrix}r-1\\k-1\end{matrix}\bigg) 
|x|^k
=
G_r(|x|)
\]
for $r\neq 0$, where $G_r(x)$ is defined by the $r$th coefficients of
\begin{equation}\label{defG_m}
\exp\bigg(\frac{xt}{1-t}\bigg)
=
1+\sum_{r=1}^{\infty} G_r(x)t^r
\end{equation}
for $|t|<1$ (see (1.2.4) in Ihara and Matsumoto~\cite{ihara-matsumoto-moscow}).
From \eqref{defG_m}, we have
\[
|H_r(x)| \leq G_r(|x|) < \exp\bigg(\frac{|x|t}{1-t}\bigg) t^{-r},
\]
for $0< t <1$.
Putting $t=1/\sqrt{1.4}$ yields 
\begin{equation}\label{G}
|H_r(ix)|<G_r(|x|)< \sqrt{1.4}^r\exp(|x|/(\sqrt{1.4}-1)).
\end{equation}
Let $\varepsilon'=\varepsilon/(3 (2c(R)^{2})^y)<1$,
we choose a integer $N_p>0$ for which
\begin{align}\label{N_p}
\sqrt{\frac{1.4}{p}}^{N_p+1}
\leq
\frac{(1-\sqrt{1.4/2})\varepsilon'}{4c(R)^2}
<
 \sqrt{\frac{1.4}{p}}^{N_p}
\end{align}
holds. 
We put
\[
\Psi_{\sigma, p}(t_p; N_p)
=
\sum_{r=0}^{N_p}\frac{H_r(ix)}{p^{\sigma r}} t_p^r
\]
which is a partial sum of $\psi_x(g_{\sigma, p}(t_p))$, 
and define
\[
\Psi_{\sigma, P_q(y)}(\bm{t}_{P_q(y)}, \bm{t}_{P_q(y)}^{-1}; \bm{N}_{P_q(y)})
=
\prod_{p\in P_q(y)}
\Psi_{\sigma, p}(t_p; N_p)
\Psi_{\sigma, p}(t_p^{-1}; N_p),
\]
where $\bm{N}_{P_q(y)}=(N_p)_{p\in P_q(y)}$.
Let
\[
h(t_p)
=
\psi_x(g_{\sigma, p}(t_p)) 
-
\Psi_{\sigma, p}(t_p; N_p).
\]
From \eqref{target1},  \eqref{G} and \eqref{N_p}, we have
\begin{align}\label{h}
|h(t_p)|
=&
\bigg|
\sum_{r=N_p+1}^{\infty} \frac{H_r(ix)}{p^{\sigma r}} t_p^r
\bigg|
\leq 
\sum_{r=N_p+1}^{\infty}\frac{G_r(R)}{p^{r/2}}
\nonumber\\
\leq& 
\exp\bigg(\frac{R}{\sqrt{1.4}-1}\bigg)
\sum_{r=N_p+1}^{\infty} \sqrt{\frac{1.4}{p}}^r
\nonumber\\
=&
c(R)
 \sqrt{\frac{1.4}{p}}^{N_p+1}
 \bigg(1-\sqrt{\frac{1.4}{p}}\bigg)^{-1}
\leq
\frac{\varepsilon'}{4c(R)}.
\end{align}
From \eqref{target1} and \eqref{defG_m}, we know
\begin{align}\label{cR}
|\psi_x(g_{\sigma, p}(t_p))|
\leq&
\sum_{r=0}^{\infty}
\frac{H_{r}(|x|)}{p^{\sigma r}}
\nonumber\\
\leq&
1+\sum_{r=1}^{\infty}
\frac{G_{r}(R)}{p^{r/2}}
=
\exp\bigg(\frac{R}{\sqrt{p}-1}\bigg)
<c(R)
\end{align}
and $|\Psi_{\sigma, p} (t_p; N_p)|<c(R)$.
Hence, for
\[
\mathcal{E}_p =
\psi_x(g_{\sigma, p}(t_p))\psi_x(g_{\sigma, p}(t_p^{-1}))
-
\Psi_{\sigma, p}(t_p; N_p)
\Psi_{\sigma, p}(t_p^{-1}; N_p),
\]
the estimates \eqref{h} and \eqref{cR} imply
\begin{align}\label{varepsilon'}
\mathcal{E}_p=
&
\Big(\Psi_{\sigma, p}(t_p; N_p)+h(t_p)\Big)
\Big(\Psi_{\sigma, p}(t_p^{-1}; N_p)+h(t_p^{-1})\Big)
\nonumber\\
& -
\Psi_{\sigma, p}(t_p; N_p)
\Psi_{\sigma, p}(t_p^{-1}; N_p)
\nonumber\\
=&
\Psi_{\sigma, p}(t_p; N_p)
h(t_p^{-1})
+
\Psi_{\sigma, p}(t_p^{-r}; N_p)
h(t_p)
+
h(t_p)h(t_p^{-1})
\nonumber\\
\leq &
\frac{\varepsilon'}{2} + \frac{{\varepsilon'}^2}{16c(R)^2}
<
\frac{\varepsilon'}{2} + \frac{\varepsilon'}{16}
<\varepsilon'.
\end{align}
From \eqref{target0}, \eqref{cR} and \eqref{varepsilon'}, we see
\begin{align}\label{apG}
&
\Big|
\psi_x(2\Re(\mathcal{G}_{\sigma, P_q(y)}(\bm{t}_{P_q(y)})))
-
\Psi_{\sigma, P_q(y)}(\bm{t}_{P_q(y)}, \bm{t}_{P_q(y)}^{-1}; \bm{N}_{P_q(y)})
\Big|
\nonumber\\
=&
\Big|
\prod_{p\in P_q(y)}
\psi_x(g_{\sigma, p}(t_p))
\psi_x(g_{\sigma, p}(t_p^{-1}))
-
\prod_{p\in P_q(y)}
\Psi_{\sigma, p}(t_p; N_p)
\Psi_{\sigma, p}(t_p^{-1}; N_p)
\Big|
\nonumber\\
=&
\Big|
\prod_{p\in P_q(y)}
\big(\Psi_{\sigma, p}(t_p; N_p)
\Psi_{\sigma, p}(t_p^{-1}; N_p)+\mathcal{E}_p\big)
\nonumber\\
& -
\prod_{p\in P_q(y)}
\Psi_{\sigma, p}(t_p; N_p)
\Psi_{\sigma, p}(t_p^{-1}; N_p)
\Big|
\nonumber\\
<&
(2^{|P_q(y)|}-1) \varepsilon'  c(R)^{2(|P_q(y)|-1)}
<
\varepsilon' (2c(R)^2)^y=\frac{\varepsilon}{3}.
\end{align}
Therefore we obtain
\begin{align}\label{key-estimate1}
&
\Big|
\frac{1}{|X(q)|}
\sum_{\chi\in X(q)}
\psi_x(2\Re(\mathcal{G}_{\sigma, P_q(y)}(\bm{\chi}_{P_q(y)})))
\nonumber
\\
&\hspace{2cm}-
\frac{1}{|X(q)|}
\sum_{\chi\in X(q)}
\Psi_{\sigma, P_q(y)}(\bm{\chi}_{P_q(y)}, \bm{\chi}_{P_q(y)}^{-1}; \bm{N}_{P_q(y)})
\Big|
<
\frac{\varepsilon}{3}.
\end{align}
%
%
\par
Secondly, 
we consider 
\[
\frac{1}{|X(q)|}
\sum_{\chi\in X(q)}
\Psi_{\sigma, P_q(y)}(\bm{\chi}_{P_q(y)}, \bm{\chi}_{P_q(y)}^{-1}; \bm{N}_{P_q(y)}).
\]
We see
\begin{align}\label{sum2'}
&
\Psi_{\sigma, P_q(y)}(\bm{\chi}_{P_q(y)}, \bm{\chi}_{P_q(y)}^{-1}; \bm{N}_{P_q(y)})
\nonumber\\
=&
\prod_{p\in P_q(y)}
\bigg(
\sum_{r=0}^{N_p}\frac{H_r(ix)}{p^{\sigma r}} \chi(p^r)
\bigg)
\bigg(
\sum_{\ell=0}^{N_p}\frac{H_{\ell}(ix)}{p^{\sigma \ell}} \overline{\chi}(p^{\ell})
\bigg)
\nonumber\\
=&
\prod_{p\in P_q(y)}
\bigg(
\sum_{r=0}^{N_p}
\frac{H_{r}^2(ix)}{p^{2\sigma r}}
+
\sum_{\ell=0}^{N_p-1}
\sum_{k=1}^{N_p-\ell}
\frac{H_{\ell}(ix)H_{\ell+k}(ix)}{p^{2\sigma \ell} p^{\sigma k}}
(\chi(p^{k})
+
\overline{\chi}(p^{k}))
\bigg)
\nonumber\\
=&
\prod_{p\in P_q(y)}
\bigg(
\sum_{r=0}^{N_p}
\frac{H_{r}^2(ix)}{p^{2\sigma r}}
\nonumber\\
&\hspace{1cm}+
\sum_{k=1}^{N_p}
\Big(
\sum_{\ell=0}^{N_p-k}
\frac{H_{\ell}(ix)H_{\ell+k}(ix)}{p^{2\sigma \ell}}
\Big)
\frac{(\chi(p^{k})
+
\overline{\chi}(p^{k}))}{p^{\sigma k}}
\bigg).
\end{align}
On the expansion of the product of $p\in P_q(y)$ in the right-hand side of \eqref{sum2'}, we denote $a(n)$ the coefficient of 
\[
\frac{1}{n^{\sigma}}
\prod_{p\mid n} 
(\chi(p^{n_p})+\overline{\chi}(p^{n_p})),
\]
where $n_p=\ord_p{n}$ which means $n$ has the prime factorisation $n=\prod_{p\mid n}p^{n_p}$. 
Then we can write
\begin{align}\label{exp}
&
\Psi_{\sigma, P_q(y)}(\bm{\chi}_{P_q(y)}, \bm{\chi}_{P_q(y)}^{-1}; \bm{N}_{P_q(y)})
\nonumber
\\
&\hspace{1.5cm}=
\prod_{p\in P_q(y)}\Big(\sum_{r=0}^{N_p} \frac{H_{r}^2(ix)}{p^{2\sigma r}}\Big)
+
\sum_{n=1}^{N}
\frac{a(n)}{n^{\sigma}}
\prod_{p\mid n} 
(\chi(p^{n_p})+\overline{\chi}(p^{n_p}))
\end{align}
and this yields
\begin{align}\label{sum2}
&
\frac{1}{|X(q)|}
\sum_{\chi\in X(q)}
\Psi_{\sigma, P_q(y)}(\bm{\chi}_{P_q(y)}, \bm{\chi}_{P_q(y)}^{-1}; \bm{N}_{P_q(y)})
-
\prod_{p\in P_q(y)}\Big(\sum_{r=0}^{N_p} \frac{H_{r}^2(ix)}{p^{2\sigma r}}\Big)
\nonumber\\
&=
\frac{1}{|X(q)|}
\sum_{n=1}^{N}
\sum_{\chi\;\mathrm{mod}\; q}
\frac{a(n)}{n^{\sigma}}\prod_{p\mid n} 
(\chi(p^{n_p})+\overline{\chi}(p^{n_p}))
\nonumber\\
&\hspace{2cm}-
\frac{1}{|X(q)|}\Psi_{\sigma, P_q(y)}(\bm{1}_{P_q(y)}, \bm{1}_{P_q(y)} ;\bm{N}_{P_q(y)})
\nonumber\\
&=
\frac{1}{|X(q)|}
\sum_{n=1}^{N}
\sum_{\chi\;\mathrm{mod}\; q}
\frac{a(n)}{n^{\sigma}}
\sum_{\substack{m_1m_2=n\\ \gcd(m_1, m_2)=1}}
\chi(m_1)\overline{\chi}(m_2) 
\nonumber\\
&\hspace{2cm}-
\frac{1}{|X(q)|}\Psi_{\sigma, P_q(y)}(\bm{1}_{P_q(y)}, \bm{1}_{P_q(y)} ;\bm{N}_{P_q(y)})
\nonumber\\
&=
\frac{q-1}{|X(q)|}
\sum_{n=1}^{N}
\frac{a(n)}{n^{\sigma}}
\sum_{\substack{m_1m_2=n\\ \gcd(m_1, m_2)=1\\m_1\equiv m_2\;\text{(mod\;$q$)}}}
1
-
\frac{1}{|X(q)|}\Psi_{\sigma, P_q(y)}(\bm{1}_{P_q(y)}, \bm{1}_{P_q(y)} ;\bm{N}_{P_q(y)}),
\end{align}
where $N=\prod_{p\in P_q(y)}p^{N_p}$ and $\bm{1}_{P_q(y)}=(\chi_0(p))_{P_q(y)}$ ($\chi_0$ is the principal character). 
If $n$ has a prime factor $p$ which satisfies $p> y$, $p=q$ or $n_p>N_p$,
then $a(n)=0$. In addition to $a(1)=0$.
Here we consider an upper bound of $N$.
Since we have
\[
\bigg(\frac{(1-\sqrt{1.4/2})\varepsilon'}{4c(R)^2}\bigg)^2
<\bigg(\frac{1.4}{p}\bigg)^{N_p}
<\bigg(\frac{\sqrt{p}}{p}\bigg)^{N_p}
=\bigg(\frac{1}{\sqrt{p}}\bigg)^{N_p}
\]
from the definition of $N_p$, 
then we see
\begin{align*}
p^{N_p/2}
&<
\bigg(\frac{4c(R)^2}{(1-\sqrt{1.4/2})\varepsilon'}\bigg)^2
=
\frac{144c(R)^4(4c(R)^4)^y}{(1-\sqrt{1.4/2})^2\varepsilon^2}
\\
&<
\bigg(\frac{576 c(R)^8}{(1-\sqrt{1.4/2})^2\varepsilon^2}\bigg)^y
=
c_1(\varepsilon, R)^{y/2},
\end{align*}
therefore we have $N< c_1(\varepsilon, R)^{y^2}$.
From \eqref{q>q_0}, we recall $N<c_1(\varepsilon, R)^{y^2}< q$,
then there is no pair of integers $(m_1, m_2)$ with $1\leq m_1m_2=n\leq N$, $m_1\equiv m_2\pmod{q}$ 
and $\gcd(m_1, m_2)=1$ except for the case $n=1$.
Since $a(1)=0$, from \eqref{q>q_0}, \eqref{cR} and \eqref{sum2}, we obtain
\begin{align}\label{key-estimate2}
&
\bigg|
\frac{1}{|X(q)|}
\sum_{\chi\in X(q)}
\Psi_{\sigma, P_q(y)}(\bm{\chi}_{P_q(y)}, \bm{\chi}_{P_q(y)}^{-1}; \bm{N}_{P_q(y)})
-
\prod_{p\in P_q(y)}\Big(\sum_{r=0}^{N_p} \frac{H_{r}^2(ix)}{p^{2\sigma r}}\Big)
\bigg|
\nonumber\\
&=
\frac{1}{|X(q)|}
\big|
\Psi_{\sigma, P_q(y)}(\bm{1}_{P_q(y)}, \bm{1}_{P_q(y)} ; \bm{N}_{P_q(y)})
\big|
<
\frac{(c(R)^2)^y}{|X(q)|}
<
\frac{\varepsilon}{6}.
\end{align}
%
%
\par
Thirdly, we consider
\[
\prod_{\substack{p \leq y \\p \neq q}}
\bigg(
\sum_{r=0}^{N_p}
\frac{H_{r}^2(ix)}{p^{2\sigma r}}
\bigg).
\]
From \eqref{apG} we have
\begin{align}\label{key-estimate3}
&
\left|
\prod_{p\in P_q(y)}
\bigg(\sum_{r=0}^{N_p}\frac{H_{r}^2(ix)}{p^{2\sigma r}}\bigg)
-
\int_{T_{P_q(y)}}
\!\!\!
\psi_x\Big(2\Re\big(\mathcal{G}_{\sigma, P_q(y)}(\bm{t}_{P_q(y)})\big)\Big)
d^{\ast}\bm{t}_{P_q(y)}
\right|
\nonumber\\
=&
\left|
\int_{T_{P_q(y)}}
\prod_{p\in P_q(y)}
\bigg(\sum_{r=0}^{N_p}\frac{H_{r}(ix)}{p^{\sigma r}} t_p^{r}\bigg)
\bigg(\sum_{\ell=0}^{N_p}\frac{H_{\ell}(ix)}{p^{\sigma \ell}} t_p^{-\ell}\bigg)
d^{\ast}\bm{t}_{P_q(y)}
\right.
\nonumber
\\
&\hspace{2cm}-
\left.
\int_{T_{P_q(y)}}
\psi_x\Big(2\Re\big(\mathcal{G}_{\sigma, P_q(y)}(\bm{t}_{P_q(y)})\big)\Big) 
d^{\ast}\bm{t}_{P_q(y)}
\right|
\nonumber
\\
= &
\left|
\int_{T_{P_q(y)}}
\Psi_{\sigma, P_q(y)}\big(\bm{t}_{P_q(y)}, \bm{t}_{P_q(y)}^{-1}; \bm{N}_{P_q(y)}\big)
d^{\ast}\bm{t}_{P_q(y)}
\right.
\nonumber\\
& \hspace{2cm}-
\left.
\int_{T_{P_q(y)}}
\psi_x\Big(2\Re\big(\mathcal{G}_{\sigma, P_q(y)}(\bm{t}_{P_q(y)})\big)\Big)
d^{\ast}\bm{t}_{P_q(y)}
\right|
\nonumber
\\
\leq &
\int_{T_{P_q(y)}}
\frac{\varepsilon}{3}
d^{\ast}\bm{t}_{P_q(y)}
=
\frac{\varepsilon}{3}.
\end{align}
%
%
\par
Finally, from \eqref{key-estimate1}, \eqref{key-estimate2} and \eqref{key-estimate3}, we obtain
\begin{align*}
&
\bigg|
\frac{1}{|X(q)|}
\sum_{\chi\in X(q)}
\psi_x\Big(2\Re\big(\mathcal{G}_{\sigma, P_q(y)}(\bm{\chi}_{P_q(y)})\big)\Big)
\\
&\hspace{1.5cm}-
\int_{T_{P_q(y)}}
\psi_x\Big(2\Re\big(\mathcal{G}_{\sigma, P_q(y)}(\bm{t}_{P_q(y)})\big)\Big)
d^{\ast}\bm{t}_{P_q(y)}
\bigg|
\\
\leq &
\bigg|
\frac{1}{|X(q)|}
\sum_{\chi\in X(q)}
\psi_x\Big(2\Re\big(\mathcal{G}_{\sigma, P_q(y)}(\bm{\chi}_{P_q(y)})\big)\Big)
\\
&\hspace{1.5cm}-
\frac{1}{|X(q)|}
\sum_{\chi\in X(q)}
\Psi_{\sigma, P_q(y)}\big(\bm{\chi}_{P_q(y)}, \bm{\chi}_{P_q(y)}^{-1}; \bm{N}_{P_q(y)}\big)
\bigg|
\\
&+
\bigg|
\frac{1}{|X(q)|}
\sum_{\chi\in X(q)}
\Psi_{\sigma, P_q(y)}\big(\bm{\chi}_{P_q(y)}, \bm{\chi}_{P_q(y)}^{-1}; \bm{N}_{P_q(y)}\big)
-
\prod_{p\in P_q(y)}\bigg(\sum_{r=0}^{N_p}\frac{H_{r}^2(ix)}{p^{2\sigma r}}\bigg)
\bigg|
\\
&+
\bigg|
\prod_{p\in P_q(y)}\bigg(\sum_{r=0}^{N_p}\frac{H_{r}^2(ix)}{p^{2\sigma r}}\bigg)
-
\int_{T_{P_q(y)}}
\psi_x\Big(2\Re\big(\mathcal{G}_{\sigma, P_q(y)}(\bm{t}_{P_q(y)})\big)\Big)
d^{\ast}\bm{t}_{P_q(y)}
\bigg|
< 
\varepsilon.
\end{align*}
\end{proof}
%
%
%
\begin{remark}
The two inequalities \eqref{N_p} mean an upper bound and a lower bound of $N_p$.
These inequalities are needed for the proof of this proposition in the case $\sigma>1/2$. If our target is just the case $\sigma>1$, we only require $N_p$ is an enough large number which satisfies \eqref{varepsilon'}. 
In the case $\sigma>1$, since $y$ and $N$ are not depend on $q$, we can see that \eqref{key-estimate2} holds for any sufficient large number $q$.
\end{remark}
For the proof of Theorem~\ref{hosoi1}, the following lemma is essential. This lemma is the case of $\Psi=\psi_x$ in Theorem~\ref{hosoi1} .
%
%
\begin{lemma}\label{main_lemma}
For $\sigma>1/2$, we have
\[
\lim_{\substack{q\to\infty\\ q \;\text{: prime}}}
\frac{1}{|X(q)|}
\sum_{\chi\in X'(q, \sigma)}
\psi_x\Big(2\Re\big(\log L(\sigma, \chi)\big)\Big)
=
\int_{\mathbb{R}}{\cal M}_{\sigma}(u)\psi_x (u) \frac{du}{\sqrt{2\pi}}
\]
in $|x|<R$ for any $R>0$.
\end{lemma}
\begin{proof}
Here we prove this lemma in the case $\sigma>1$. The proof of the case $1\geq \sigma > 1/2$ is in the next section.
Since $\sigma>1$, we know $X'(q, \sigma)=X(q)$.
For any $\varepsilon>0$,  let $y$ be a large number satisfying
\begin{equation}\label{estimate>1_1}
\Big|2\Re\big(\log L(\sigma, \chi)\big)
-
2\Re\big(\mathcal{G}_{\sigma, P_q(y)}(\bm{\chi}_{P_q(y)})\big)
\Big|
<
\frac{\varepsilon}{3R}
\end{equation}
and
\begin{equation}\label{estimate>1_2}
|\widetilde{\mathcal{M}}_{\sigma, P_q(y)}(x)
-\widetilde{\mathcal{M}}_{\sigma}(x)|
<
\frac{\varepsilon}{3},
\end{equation}
where $y$ does not depend on $q$.
From Proposition~\ref{matsumoto-umegaki-ijnt_prop1} and Lemma~\ref{key_lemma_hosoi} , 
we can find an integer $q_0$ such that any prime number $q$ with $q>q_0$ satisfies
\begin{equation}\label{estimate>1_3}
\Bigg|
\frac{1}{|X(q)|}
\sum_{\chi\in X(q)}
\psi_x\Big(2\Re\big(\mathcal{G}_{\sigma, P_q(y)}(\bm{\chi}_{P_q(y)})\big)\Big)
-\int_{\mathbb{R}}\mathcal{M}_{\sigma, P_q(y)}(u)\psi_x(u)\frac{du}{\sqrt{2\pi}}
\Bigg|
<
\frac{\varepsilon}{3}.
\end{equation}
Since
\begin{equation}\label{psi}
|\psi_x(u)-\psi_x(u')|
\leq |x|\cdot |u-u'|
\end{equation}
(see (6.5.19) in \cite{ihara} or (97) in\cite{ihara-matsumoto-QJM}), 
we have
\begin{align*}
&
\Big|
\frac{1}{|X(q)|}
\sum_{\chi\in X(q)}
\psi_x(2\log |L(\sigma, \chi)|)
-
\int_{\mathbb{R}}\mathcal{M}_{\sigma}(u)\psi_x (u) \frac{du}{\sqrt{2\pi}}\Big|
\\
\leq &
\Big|
\frac{1}{|X(q)|}
\sum_{\chi\in X(q)}
\psi_x(2\log |L(\sigma, \chi)|)
-
\frac{1}{|X(q)|}
\sum_{\chi\in X(q)}
\psi_x(2\Re(\mathcal{G}_{\sigma, P_q(y)}(\bm{\chi}_{P_q(y)})))
\Big|
\\
&+
\Big|
\frac{1}{|X(q)|}
\sum_{\chi\in X(q)}
\psi_x(2\Re(\mathcal{G}_{\sigma, P_q(y)}(\bm{\chi}_{P_q(y)})))
-
\int_{\mathbb{R}}\mathcal{M}_{\sigma, P_q(y)}(u)\psi_x (u) \frac{du}{\sqrt{2\pi}}\Big|
\\
&+
\Big|
\int_{\mathbb{R}}\mathcal{M}_{\sigma, P_q(y)}(u)\psi_x (u) \frac{du}{\sqrt{2\pi}}
-
\int_{\mathbb{R}}\mathcal{M}_{\sigma}(u)\psi_x (u) \frac{du}{\sqrt{2\pi}}\Big|
\\
\leq &
\frac{2|x|}{|X(q)|}
\sum_{\chi\in X(q)}
\big|\log |L(\sigma, \chi)|
-
\Re(\mathcal{G}_{\sigma, P_q(y)}(\bm{\chi}_{P_q(y)}))\big|
\\
&+
\Big|
\frac{1}{|X(q)|}
\sum_{\chi\in X(q)}
\psi_x(\Re(\mathcal{G}_{\sigma, P_q(y)}(\bm{\chi}_{P_q(y)})))
-
\int_{\mathbb{R}}\mathcal{M}_{\sigma, P_q(y)}(u)\psi_x (u) \frac{du}{\sqrt{2\pi}}\Big|
\\
&+
\Big|
\int_{\mathbb{R}}\mathcal{M}_{\sigma, P_q(y)}(u)\psi_x (u) \frac{du}{\sqrt{2\pi}}
-
\int_{\mathbb{R}}\mathcal{M}_{\sigma}(u)\psi_x (u) \frac{du}{\sqrt{2\pi}}\Big|
\\
=:&
\mathcal{X}_q +\mathcal{Y}_q +\mathcal{Z}_q,
\end{align*}
say.
From \eqref{estimate>1_1}, \eqref{estimate>1_2} and \eqref{estimate>1_3}, we see $\mathcal{X}_q+\mathcal{Y}_q +\mathcal{Z}_q<\varepsilon$.
\end{proof}
%
%
\section{The proof of Lemma~\ref{main_lemma} and Theorem~\ref{hosoi1}}
\par
In this section, 
we prove Lemma~\ref{main_lemma} 
for $1/2 < \sigma \leq 1$ and $|x|<R$. 
Let $\delta=16(\sigma-1/2)$ and $Y=q$ for Assumption~\ref{ass}.
By \eqref{montgomery}, we see
\[
\mathcal{V}_q:=
\frac{1}{|X(q)|}
\sum_{\chi\in X(q)\setminus X_{\mathscr{D}}(q)} 1
\ll
\frac{q^{A(\sigma)}(\log q)^{14}}{|X(q)|}.
\]
Since $|X(q)|=q-2$ and $A(\sigma)<1$, 
we see $\mathcal{V}_q\to 0$ as $q\to \infty$.
For the proof of Lemma~\ref{main_lemma}, we consider
\[
\Big|
\frac{1}{|X(q)|}
\sum_{\chi\in X'(q, \sigma)}
\psi_x(2\log |L(\sigma, \chi)|)
-
\int_{\mathbb{R}}\mathcal{M}_{\sigma}(u)\psi_x (u) \frac{du}{\sqrt{2\pi}}\Big|.
\]
Let $X'_1(q, \sigma)=X'(q, \sigma)\cap X_{\mathscr{D}}(q)$ and 
$X'_2(q, \sigma)=X'(q, \sigma)\setminus X'_1(q, \sigma)$.
We divide the sum over $X'(q, \sigma)$ into two summations.
One is the sum over $\chi\in X'_1(q, \sigma)$,  
and the other is the sum over $X'_2(q, \sigma)$.
We know
\begin{align*}
\mathcal{W}_q:=&
\frac{1}{|X(q)|}
\sum_{\chi\in X(q)\setminus X'_1(q, \sigma)} 1
\\
=&
\frac{|X(q)|-|X'(q, \sigma)|+|X'_2(q, \sigma)|}{|X(q)|}
\leq 
1-\frac{|X'(q, \sigma)|}{|X(q)|} + \mathcal{V}_q.
\end{align*}
Since Corollary~2.2 in \cite{ihara-matsumoto-QJM}, we see
$\mathcal{W}_q\to 0$ as $q\to\infty$. 
By \eqref{psi}, we see
\begin{align*}
&
\Big|
\frac{1}{|X(q)|}
\sum_{\chi\in X'(q, \sigma)}
\psi_x(2\log |L(\sigma, \chi)|)
-
\int_{\mathbb{R}}\mathcal{M}_{\sigma}(u)\psi_x (u) \frac{du}{\sqrt{2\pi}}\Big|
\\
\leq &
\Big|
\frac{1}{|X(q)|}
\sum_{\chi\in X'_2(q, \sigma)}
\psi_x(2\log |L(\sigma, \chi)|)
\Big|
\\
&+
\Big|
\frac{1}{|X(q)|}
\sum_{\chi\in X'_1(q, \sigma)}
\psi_x(2\log |L(\sigma, \chi)|)
-
\frac{1}{|X(q)|}
\sum_{\chi\in X'_1(q, \sigma)}
\psi_x(2\Re(\mathcal{G}_{\sigma, P_q(y)}(\bm{\chi}_{P_q(y)})))
\Big|
\\
&+
\Big|
\frac{1}{|X(q)|}
\sum_{\chi\in X'_1(q, \sigma)}
\psi_x(2\Re(\mathcal{G}_{\sigma, P_q(y)}(\bm{\chi}_{P_q(y)})))
-
\int_{\mathbb{R}}\mathcal{M}_{\sigma, P_q(y)}(u)\psi_x (u) \frac{du}{\sqrt{2\pi}}
\Big|
\\
&+
\Big|
\int_{\mathbb{R}}\mathcal{M}_{\sigma, P_q(y)}(u)\psi_x (u) \frac{du}{\sqrt{2\pi}}
-
\int_{\mathbb{R}}\mathcal{M}_{\sigma}(u)\psi_x (u) \frac{du}{\sqrt{2\pi}}\Big|
\\
\leq&
\mathcal{V}_q
+
\frac{2|x|}{|X(q)|}
\sum_{\chi\in X'_1(q, \sigma)}
\Big|
\log |L(\sigma, \chi)|
-
\Re(\mathcal{G}_{\sigma, P_q(y)}(\bm{\chi}_{P_q(y)}))
\Big|
\\
&+
\Big|
\frac{1}{|X(q)|}
\sum_{\chi\in X(q)}
\psi_x(2\Re(\mathcal{G}_{\sigma, P_q(y)}(\bm{\chi}_{P_q(y)})))
-
\int_{\mathbb{R}}\mathcal{M}_{\sigma, P_q(y)}(u)\psi_x (u) \frac{du}{\sqrt{2\pi}}\Big|
\\
&+
\Big|
\frac{-1}{|X(q)|}
\sum_{\chi\in X(q)\setminus X'_1(q, \sigma)}
\psi_x(2\Re(\mathcal{G}_{\sigma, P_q(y)}(\bm{\chi}_{P_q(y)})))
\Big|
\\
&+
\Big|
\int_{\mathbb{R}}\mathcal{M}_{\sigma, P_q(y)}(u)\psi_x (u) \frac{du}{\sqrt{2\pi}}
-
\int_{\mathbb{R}}\mathcal{M}_{\sigma}(u)\psi_x (u) \frac{du}{\sqrt{2\pi}}\Big|
\\
\leq&
\mathcal{V}_q
+
\frac{2|x|}{|X(q)|}
\sum_{\chi\in X'_1(q, \sigma)}
\Big|
\log |L(\sigma, \chi)|
-
\Re(\mathcal{G}_{\sigma, P_q(y)}(\bm{\chi}_{P_q(y)}))
\Big|
+\mathcal{Y}_q +\mathcal{W}_q+\mathcal{Z}_q,
\end{align*}
say.
Let $y=\sqrt{\log\log q}$ for large $q$. 
We see $\mathcal{Y}_q \to 0$ as $q \to \infty$ by Lemma~\ref{key_lemma_hosoi} and Proposition~\ref{matsumoto-umegaki-ijnt_prop1}, 
and we see $\mathcal{Z}_q \to 0$ as $q \to \infty$ by Proposition~\ref{matsumoto-umegaki-ijnt_prop3}. 
In above, we mentioned $\mathcal{V}_q \to 0$ and $\mathcal{W}_q \to 0$ as $q\to\infty$.
Therefore, in this section, we will show
\[
\mathcal{X}'_q 
:=
\frac{2|x|}{|X(q)|}
\sum_{\chi\in X'_1(q, \sigma)}
\Big|
\log |L(\sigma, \chi)|
-
\Re(\mathcal{G}_{\sigma, P_q(y)}(\bm{\chi}_{P_q(y)}))
\Big|
\to 0
\]
as $q \to \infty$ 
by the method in \cite{matsumoto-umegaki-ijnt}.
\begin{proposition}\label{mu}
Let $Y=q$ be a large number (here, $q$ does not need to be a prime number) and we assume that $L(s, \chi)$ satisfies Assumption~\ref{ass}, where $\chi$ is a primitive character of conductor $q$.
For any $0< \varepsilon <1/2$, we put $1/2+2\varepsilon \leq u_0<3/2$. Then we have
\begin{equation}\label{F}
\Re(\log L(u_0, \chi) - \mathcal{G}_{u_0, P_q(y)}(\bm{\chi}_{P_q(y)}) 
-
S_y)
\ll_{\varepsilon} (\log\log q)^{-\varepsilon},
\end{equation}
where
\[
S_y
=
\sum_{\substack{p>y\\ p\neq q}} 
\frac{\chi(p) e^{-p/q}}{p^{u_0}}
\]
and $y=\sqrt{\log\log q}$.
\end{proposition}
\begin{proof} 
We consider
\[
L_{P_q(y)}(s,\chi)
=
\prod_{p\in P_q(y)} (1-\chi(p)p^{-s})^{-1}
\]
and put
\[
F(s, \chi)
:=\frac{L(s,\chi)}{L_{P_q(y)}(s,\chi)}.
\]
For $\sigma>1$, we see 
\[
\log F(s, \chi)
=\log L(s, \chi) - \log L_{P_q(y)}(s,\chi) 
= 
\sum_{p>y}\sum_{k=1}^{\infty}\frac{\chi(p^k)}{kp^{ks}}
\]
and
\[
\frac{F'}{F}(s, \chi)
=
-\sum_{p>y}\sum_{k=1}^{\infty}\frac{\chi(p^k)\log p}{p^{ks}}.
\]
From \eqref{for_th1}, we have
\[
\frac{L'(s, \chi)}{L(s, \chi)} \ll_{\varepsilon} \log q,
\]
for $1/2+\varepsilon \leq \sigma \leq 1$ and $|t| \leq \log q$.
And we know
\begin{align*}
\bigg|
\frac{L'_{P_q(y)}(s, \chi)}{L_{P_q(y)}(s, \chi)}
\bigg|
\leq&
\sum_{p\leq y}\sum_{k=1}^{\infty}
\frac{\log p}{p^{k\sigma}}
\leq
\log y 
\sum_{p\leq y}\sum_{k=1}^{\infty}
\frac{1}{p^{k\sigma}}
\\
=&
\log y
\sum_{p\leq y}
\frac{1}{p^{\sigma}-1}
\ll
y.
\end{align*}
Hence
\[
\frac{F'(s, \chi)}{F(s, \chi)} 
\ll_{\varepsilon} 
\log q +y
\ll
\log q
\]
for $1/2+\varepsilon \leq \sigma \leq 1$ and $|t|\leq \log q$.
Let $1/2+2\varepsilon \leq u \leq 3/2$ and $X>1$,  Mellin's formula 
\begin{equation}\label{mellin_c}
e^{-w}=\frac{1}{2\pi i}\int_{(c)} w^{-s}\Gamma(s) ds
\end{equation}
yields
\[
-\sum_{p>y}
\sum_{k=1}^{\infty}\frac{\chi(p^k)\log p}{p^{k u}} e^{-p^k/X}
=
\frac{1}{2\pi i}\int_{(1/2)}
\frac{F'}{F}(u+s, \chi) X^s \Gamma (s) ds,
\]
where the path of integration $(c)$ means the vertical line $\Re(z)=c>0$.
By the residue theorem, we see 
\begin{align*}
&
-\sum_{p>y} 
\sum_{k=1}^{\infty}\frac{\chi(p^k)\log p}{p^{k u}} e^{-p^k/X}
-
\frac{F'}{F}(u, \chi) 
\\
= &
\frac{1}{2\pi i}\int_{\substack{1/2+\varepsilon-u\leq \Re(s) \leq 1/2 \\ |\Im(s)|= \log q}}
+\frac{1}{2\pi i} \int_{\substack{\Re(s)=1/2+\varepsilon-u\\ |\Im(s)|\leq \log q}} 
+\frac{1}{2\pi i} \int_{\substack{\Re(s)=1/2 \\ |\Im(s)|\geq \log q}},
\end{align*}
where the integrands are $F'/F(u+s, \chi)X^s\Gamma(s)$.
By the estimate
\begin{equation}\label{stirling}
\Gamma(s)\ll |t|^{\sigma -1/2}e^{-\pi|t|/2}
\quad
(|\sigma|\leq 1,\; |t|>1),
\end{equation}
which is shown by Stirling's formula, 
we obtain
\[
-\sum_{p\neq q}
\sum_{k=1}^{\infty}\frac{\chi(p^k)\log p}{p^{k u}} e^{-p^k/X}
-
\frac{F'}{F}(u, \chi) 
\ll
\frac{\log q}{X^{\varepsilon}}
+\frac{X^{1/2} \log q}{q^{\pi/2}}.
\]
Let $u_0$ be a real number with $1/2+2\varepsilon\leq u_0<3/2$.
Considering the integration of $u$ in the above formula from $u_0$ to $3/2$ yields
\begin{align*}
&
- \log F(3/2, \chi) + \log F(u_0, \chi)  
\\
=&
\int_{u_0}^{3/2}
\sum_{p>y} 
\sum_{k=1}^{\infty}\frac{\chi(p^k)\log p}{p^{k u}} e^{-p^k/X}
du
+
O\bigg(
\frac{\log q}{X^{\varepsilon}}
+\frac{X^{1/2} \log q}{q^{\pi/2}}
\bigg).
\end{align*}
The integral in this formula is
\begin{align*}
&
\sum_{\substack{p>y\\ p\neq q}} 
\log p 
\sum_{k=1}^{\infty} \chi(p^k) e^{-p^k/X}
\int_{u_0}^{3/2}
p^{-k u} du
\\
=&
\sum_{\substack{p>y\\ p\neq q}} 
\log p 
\sum_{k=1}^{\infty} \chi(p^k) e^{-p^k/X}
\bigg[\frac{-p^{-ku}}{k\log p}\bigg]_{u_0}^{3/2}
\\
=&
\sum_{\substack{p>y\\ p\neq q}} 
\frac{\chi(p) e^{-p/X}}{p^{u_0}}
+
\sum_{\substack{p>y\\ p\neq q}} 
\sum_{k=2}^{\infty}  
\frac{\chi(p^k)e^{-p^k/X}}{kp^{ku_0}}
-
\sum_{p>y}\sum_{k=1}^{\infty}
\frac{\chi(p^k)e^{-p^k/X}}{kp^{3k/2}}.
\end{align*}
In the right-hand side of the above equation,
the last two sums are each estimated by
\[
\sum_{p>y} 
\sum_{k=2}^{\infty} \frac{\chi(p^k)e^{-p^k/X}}{k p^{k u_0}} 
<
\sum_{k=2}^{\infty}
\sum_{p>y} 
\frac{1}{k p^{k (1/2+2\varepsilon)}} 
\ll
\sum_{k=2}^{\infty} \frac{1}{k^2}
\sum_{p>y} 
\frac{1}{p^{k/3+2k\varepsilon}} 
\ll_{\varepsilon} 
 \frac{1}{y^{2\varepsilon}}
\]
and
\[
\sum_{p>y}\sum_{k=1}^{\infty}
\frac{\chi(p^k)e^{-p^k/X}}{kp^{3k/2}}
\ll 
\sum_{p>y}\sum_{k=1}^{\infty}
\frac{1}{k p^{3k/2}}
<
\sum_{p>y}\sum_{k=1}^{\infty}
\frac{1}{k^2 p^{4k/3} }
\ll
\sum_{p>y} \frac{1}{p^{4/3}}
\ll
\frac{1}{y^{1/6}}.
\]
We see also $\log F(3/2, \chi) \ll y^{-1/6}$, 
then we have
\begin{equation}\label{Fu_0}
\log F(u_0, \chi)
=
\sum_{\substack{p>y\\ p\neq q}} 
\frac{\chi(p) e^{-p/X}}{p^{u_0}}
+
O_{\varepsilon}\bigg(
\frac{1}{y^{2\varepsilon}}
+
\frac{\log q}{X^{\varepsilon}}
+\frac{X^{1/2} \log q}{q^{\pi/2}}
\bigg).
\end{equation}
Since
\[
\Re(\log {L}_{P_q(y)}(u_0, \chi))
=
\Re(\mathcal{G}_{u_0, P_q(y)}(\bm{\chi}_{P_q(y)})),
\]
putting $X=q$ and \eqref{Fu_0} yield
\begin{equation}\label{sigma<1}
\Re(\log L(u_0, \chi) - \mathcal{G}_{u_0, P_q(y)}(\bm{\chi}_{P_q(y)}) 
-
S_y)
\ll_{\varepsilon}
\frac{1}{y^{2\varepsilon}}+\frac{\log q}{q^{\varepsilon}}.
\end{equation}
By $y=\sqrt{\log\log q}$, we have
\[
\Re(\log L(u_0, \chi) - \mathcal{G}_{u_0, P_q(y)}(\bm{\chi}_{P_q(y)}) 
-
S_y)
\ll_{\varepsilon} (\log\log q)^{-\varepsilon}.
\]
\end{proof}
Since we see
\begin{align}\label{S_y}
&
\frac{1}{|X(q)|}
\sum_{\chi\in X'(q, \sigma)}
|\Re(S_y)|
\leq
\frac{1}{|X(q)|}
\sum_{\chi\;\text{mod}\; q}
|S_y|
\nonumber
\\
\leq&
\frac{1}{|X(q)|}
\bigg(\sum_{\chi\; \text{mod}\; q} 1\bigg)^{1/2}
\bigg(\sum_{\chi\; \text{mod}\; q} |S_y|^2\bigg)^{1/2}
\nonumber
\\
\ll&
\frac{1}{\sqrt{q}}
\bigg(
\sum_{\chi\; \text{mod}\; q} 
\bigg(
\sum_{\substack{p>y\\ p\neq q}} 
\frac{e^{-2p/q}}{p^{1+4\varepsilon}}
+
\sum_{\substack{p_i>y\\ p_i\neq q\\ p_1\neq p_2}} 
\frac{\chi(p_1p_2^{-1}) e^{-(p_1+p_2)/q}}{p_1^{u_0}p_2^{u_0}}
\bigg)
\bigg)^{1/2}
\nonumber
\\
\ll_{\varepsilon}&
\bigg(
\frac{1}{y^{2\varepsilon}}
+
\sum_{\substack{p_i>y\\ p_i\neq q\\ p_1\equiv p_2 \;\text{mod}\; q}} 
\frac{e^{-(p_1+p_2)/q}}{p_1^{u_0}p_2^{u_0}}
\bigg)^{1/2}
\nonumber
\\
\ll_{\varepsilon} &
\bigg(
\frac{1}{y^{2\varepsilon}}
+
\sum_{p_1>y}
\frac{e^{-p_1/q}}{p_1^{u_0}}
\sum_{\substack{p_2>p_1\\ p_1\equiv p_2 \;\text{mod}\; q}} 
\frac{e^{-p_2/q}}{p_2^{u_0}}
\bigg)^{1/2}
\nonumber
\\
\leq&
\bigg(
\frac{1}{y^{2\varepsilon}}
+
\sum_{p_1>y}
\frac{e^{-p_1/q}}{p_1^{2u_0}}
\sum_{\substack{p_2>p_1\\ p_1\equiv p_2 \;\text{mod}\; q}} 
e^{-p_2/q}
\bigg)^{1/2}
\ll
\frac{1}{y^{\varepsilon}}=\frac{1}{(\log\log q)^{\varepsilon/2}},
\end{align}
from Proposition~\ref{mu} and \eqref{S_y}, we have
\[
\mathcal{X}'_q
\ll
\frac{1}{(\log\log q)^{\varepsilon/2}} \to 0
\]
as $q\to\infty$.
The proof of Lemma~\ref{main_lemma} is thus complete. 
We can obtain Theorem~\ref{hosoi1} immediately by the same argument in \cite{ihara-matsumoto-QJM}, \cite{ihara-matsumoto-moscow}, \cite{matsumoto} and \cite{matsumoto-umegaki-ijnt}.
%
%
\section{The preparation of the proof of Theorem~\ref{main1}}
\par
We recall $\mathcal{L}(s, \chi)$ is either $\log L(s, \chi)$ or $L'(s, \chi)/L(s, \chi)$. 
By Ihara and Matsumoto's work \cite{ihara-matsumoto-moscow}, we can write
\[
\psi_x(\mathcal{L}(s, \chi))
=\sum_{n=1}^{\infty}\frac{\lambda_{2x}(n)\chi(n)}{n^s},
\]
where
\[
\lambda_{2x}(n)=\prod_{p\mid n} \lambda_{2x}(p^{n_p})
\]
for $n=\prod_p p^{n_p}$ and
\[
\lambda_{2x}(p^{n_p})=
\begin{cases}
G_{n_p}(-ix\log p) & \;\text{if}\;\mathcal{L}(s, \chi)=L'(s, \chi)/L(s, \chi),
\\
H_{n_p}(-ix) & \;\text{if}\;\mathcal{L}(s, \chi)=\log L(s, \chi).
\end{cases}
\]
In the part of Theorem~2 in \cite{ihara-matsumoto-moscow}, we see $|\lambda_{2x}(n)|\ll_{\varepsilon} n^{\varepsilon}$ for any $\varepsilon>0$.
In this section, we will prove the following proposition.
\begin{proposition}\label{key_prop}
For $\sigma>1/2$, we have
\[
\lim_{Y\to\infty}\frac{1}{N (Y)}
\underset{|D|\leq Y}{{\sum}^{\dag}}
\psi_x(\mathcal{L}(\sigma, \chi_D))
=
\sum_{n=1}^{\infty}
\frac{\lambda_{2x}(n^2)\prod_{p\mid n}p(p+1)^{-1}}{n^{2\sigma}}.
\]
\end{proposition}
For the proof of this proposition, we show
\begin{lemma}\label{asymptotic}
Let $Y$ be an sufficiently large number and $\chi$ a primitive Dirichlet character of conductor $q\leq Y$. 
For $1/2+\delta_1 \leq \sigma \leq 5/4$ with $0<\delta_1 <1/8$ and $X>1$, under Assumption~\ref{ass} with $\delta=8\delta_1$,  
there exist positive constants $c(\delta_1)$ and $a(\delta_1)$ with $0<a(\delta_1)<1$ such that
\[
\psi_x(\mathcal{L}(\sigma, \chi))
=
\sum_{n=1}^{\infty}
\frac{\lambda_{2x}(n)\chi(n)}{n^{\sigma}}e^{-n/X}
+
O\bigg(\Big(\frac{X^{1/2}}{Y^{\pi/2}}+\frac{1}{X^{\delta_1/2}}\Big)Y^{R c(\delta_1)\varepsilon_0^{(1-a(\delta_1))}}\bigg),
\]
where $|x|<R$ for any $R>0$.
\end{lemma}
\begin{proof}
For $z=u+iv$, Mellin's formula \eqref{mellin_c} yields
\[
\sum_{n=1}^{\infty}
\frac{\lambda_{2x}(n)\chi(n)}{n^{\sigma}}e^{-n/X}
=
\frac{1}{2\pi i}
\int_{(1/2)}\psi_x(\mathcal{L}(\sigma+z, \chi)) X^z\Gamma(z)dz.
\]
Since we assume Assumption~\ref{ass}, 
by the residue theorem, we have
\begin{align}\label{int}
\sum_{n=1}^{\infty}
\frac{\lambda_{2x}(n)\chi(n)}{n^{\sigma}}e^{-n/X}
=&
\psi_x(\mathcal{L}(\sigma, \chi))
+
\frac{1}{2\pi i}\int_{\substack{|v|\geq\log Y\\ u=1/2}}
\psi_x(\mathcal{L}(\sigma+z, \chi)) X^z\Gamma(z)dz
\nonumber\\
&+
\frac{1}{2\pi i}\int_{\substack{|v|=\log Y \\ -\delta_1/2\leq u \leq 1/2}}
\psi_x(\mathcal{L}(\sigma+z, \chi)) X^z\Gamma(z)dz
\nonumber\\
&+
\frac{1}{2\pi i}\int_{\substack{|v|\leq\log Y \\ u=-\delta_1/2}}
\psi_x(\mathcal{L}(\sigma+z, \chi)) X^z\Gamma(z)dz
\nonumber\\
=&
\psi_x(\mathcal{L}(\sigma, \chi))
+
I_1+I_2+I_3,
\end{align}
say. On the path of integration of $I_1$, since 
\[
|\psi_x(\mathcal{L}(\sigma+z))|
<
\exp(-x\arg\mathcal{L}(\sigma+z))
<
\exp(R|\mathcal{L}(\sigma+z)|)\ll_{\delta_1, R} 1
\]
and \eqref{stirling}, we have
\[
I_1
\ll_{\delta_1, R} 
\int_{\log Y}^{\infty}
X^{1/2}e^{-\pi v/2} dv
\ll
X^{1/2}Y^{-\pi/2}.
\]
By Lemma~\ref{arg} and \eqref{stirling}, we see
\[
I_2
\ll
Y^{R c(\delta_1) \varepsilon_0^{1-a(\delta_1)}} X^{1/2} Y^{-\pi/2} 
\]
and
\[
I_3
\ll
Y^{R c(\delta_1) \varepsilon_0^{1-a(\delta_1)}} X^{-\delta_1/2}.
\]
\end{proof}
\begin{proof}[Proof of Proposition~\ref{key_prop}]
We recall that $N(Y)$ is the number of fundamental discriminants $D$ with $|D|\leq Y$, and we have
\begin{equation}\label{N_Y}
N(Y)=\underset{|D|\leq Y}{{\sum}^{\ast}} 1=\frac{6}{\pi^2} Y+O(Y^{1/2})
\end{equation}
(see \cite{mourtada-murty}). 
We see
\begin{align}\label{tg}
&
\frac{1}{N(Y)}
\underset{|D|\leq Y}{{\sum}^{\dag}}
\psi_x(\mathcal{L}(\sigma, \chi_D))
\nonumber\\
=&
\frac{1}{N(Y)}
\underset{|D|\leq Y}{{\sum}^{\dag_1}}
\psi_x(\mathcal{L}(\sigma, \chi_D))
+
\frac{1}{N(Y)}
\underset{|D|\leq Y}{{\sum}^{\dag_2}}
\psi_x(\mathcal{L}(\sigma, \chi_D)),
\end{align}
where $\sum^{\dag_1}$ is the sum over $D$ which satisfies the condition that $L(s, \chi_D)$ does not have a zero in $\mathscr{D}$ (this means Assumption~\ref{ass} holds) and $\sum^{\dag_2}$ is the remaining. 
\par
At first, we consider the case $1/2+\delta_1 \leq \sigma \leq 5/4$ with $0<\delta_1 \leq 1/8$. 
Since $L(\sigma, \chi_D)>0$ in ${\sum}^{\dag}$, by \eqref{jutila} and \eqref{N_Y} with $\varepsilon=\delta_1/2$, we know
\begin{align}\label{sum_dag2}
\Big|
\frac{1}{N(Y)}
\underset{|D|\leq Y}{{\sum}^{\dag_2}}
\psi_x(\mathcal{L}(\sigma, \chi_D))
\Big|
\leq&
\frac{1}{N(Y)}
\underset{|D|\leq Y}{{\sum}^{\dag_2}}
|\psi_x(\mathcal{L}(\sigma, \chi_D))|
\leq
\frac{1}{N(Y)}
\underset{|D|\leq Y}{{\sum}^{\dag_2}} 1
\nonumber\\
\ll&
\frac{(Y\log Y)^{(7-6\sigma)/(6-4\sigma)+\varepsilon}}{N(Y)}
\nonumber\\
\ll&
Y^{(1-2\sigma)/(6-4\sigma)+\varepsilon}
(\log Y)^{(7-6\sigma)/(6-4\sigma)+\varepsilon}
\nonumber\\
\ll_\varepsilon&
Y^{-2\delta_1+2\varepsilon}
\ll_{\delta_1} Y^{-\delta_1}.
\end{align}
From  Lemma~\ref{asymptotic}, the equation \eqref{tg} and \eqref{sum_dag2}, we have
\begin{align*}
\frac{1}{N(Y)}
\underset{|D|\leq Y}{{\sum}^{\dag}}
\psi_x(\mathcal{L}(\sigma, \chi_D))
=&
\sum_{n=1}^{\infty}
\frac{\lambda_{2x}(n)e^{-n/X}}{n^{\sigma}}
\bigg(
\frac{1}{N(Y)}
\underset{|D|\leq Y}{{\sum}^{\dag_1}}\chi_D(n)
\bigg)
\\
&
+O\bigg(\Big(\frac{X^{1/2}}{Y^{\pi/2}}+\frac{1}{X^{\delta_1/2}}\Big)Y^{R c(\delta_1)\varepsilon_0^{(1-a(\delta_1))}}
+Y^{-\delta_1}\bigg)
\\
=&
\sum_{n=1}^{\infty}
\frac{\lambda_{2x}(n)e^{-n/X}}{n^{\sigma}}
\bigg(
\frac{1}{N(Y)}
\underset{|D|\leq Y}{{\sum}^{\ast}} \chi_D(n)
\bigg)
\\
&
+
O\bigg(
\sum_{n=1}^{\infty}
\frac{|\lambda_{2x}(n)|X^{1/2}}{n^{\sigma+1/2}}
\bigg(
\frac{1}{N(Y)}
\underset{|D|\leq Y}{{\sum}^{\ddag}} 1
\bigg)
\bigg)
\\
&
+O\bigg(
\Big(\frac{X^{1/2}}{Y^{\pi/2}}+\frac{1}{X^{\delta_1/2}}\Big)Y^{R c(\delta_1)\varepsilon_0^{(1-a(\delta_1))}}+Y^{-\delta_1}\bigg),
\end{align*}
where $\sum^{\ddag}={\sum}^{\ast}-{\sum}^{\dag_1}$.
If $D$ appears in $\sum^{\ddag}$, 
then $\chi_D$ does not satisfy Assumption~\ref{ass}.
Hence, by the calculations similar to \eqref{sum_dag2}, 
putting $X=Y^{\delta_1}$ yields
\[
\sum_{n=1}^{\infty}
\frac{|\lambda_{2x}(n)|X^{1/2}}{n^{\sigma+1/2}}
\bigg(
\frac{1}{N(Y)}
\underset{|D|\leq Y}{{\sum}^{\ddag}} 1
\bigg)
\ll_{\delta_1}
X^{1/2}Y^{-\delta_1}
=
Y^{-\delta_1/2}.
\] 
Then we obtain
\begin{align*}
&
\frac{1}{N(Y)}
\underset{|D|\leq Y}{{\sum}^{\dag}}
\psi_x(\mathcal{L}(\sigma, \chi_D))
\\
=&
\frac{1}{N(Y)}
\sum_{n=1}^{\infty}
\frac{\lambda_{2x}(n)e^{-n/X}f_Y(n)}{n^{\sigma}}
+O\Big(Y^{R c(\delta_1)\varepsilon_0^{(1-a(\delta_1))}-\delta_1^2/2}+Y^{-\delta_1}\Big),
\end{align*}
where
\[
f_Y(n)=\underset{|D|\leq Y}{{\sum}^{\ast}}\chi_D(n).
\]
We can take a suitable $\varepsilon_0>0$ which satisfies $R c(\delta_1) \varepsilon_0^{(1-a(\delta_1))}<\delta_1^2/4$  (we take a suitable large $Y$ accordingly) in advance. 
Hence we obtain
\begin{align*}
&
\frac{1}{N(Y)}
\underset{|D|\leq Y}{{\sum}^{\dag}}
\psi_x(\mathcal{L}(\sigma, \chi_D))
=
\frac{1}{N(Y)}
\sum_{n=1}^{\infty}
\frac{\lambda_{2x}(n)e^{-n/X}f_Y(n)}{n^{\sigma}}
+
O(Y^{-\delta_1^2/4})
\nonumber\\
=&
\frac{1}{N(Y)}
\sum_{\substack{n=1\\ \text{square}}}^{\infty}
\frac{\lambda_{2x}(n)e^{-n/X}f_Y(n)}{n^{\sigma}}
\nonumber
\\
&+
\frac{1}{N(Y)}
\sum_{\substack{n=1\\ \text{non-sq}}}^{\infty}
\frac{\lambda_{2x}(n)e^{-n/X}f_Y(n)}{n^{\sigma}}
+
O(Y^{-\delta_1^2/4}),
\end{align*}
where $\sum_{\text{square}}$ is the sum over square integers and $\sum_{\text{non-sq}}$ is the sum over non-square integers.
We will calculate these summations by \eqref{N_Y},
\begin{equation}\label{non-sq}
\sum_{\substack{n=1\\ \text{non-sq}}}^N
|f_Y(n)|^2 \ll YN(\log N)^4
\end{equation}
and
\begin{align}\label{sq}
f_Y(n^2)
=&
\frac{6}{\pi^2}\prod_{p\mid n}\bigg(1+\frac{1}{p}\bigg)^{-1}Y+O(Y^{1/2}d(n^2))
\nonumber
\\
=&
N(Y)\prod_{p\mid n}\bigg(1+\frac{1}{p}\bigg)^{-1}+O(Y^{1/2}d(n^2)),
\end{align}
where $d(n)$ is the divisor function which means the number of positive divisors of $n$.
These facts are in Mourtada and Murty~\cite{mourtada-murty},  and the following calculations are analogous to the calculations in \cite{mourtada-murty}.
By \eqref{non-sq}, we have
\begin{align*}
&
\sum_{\substack{n=1\\ \text{non-sq}}}^{\infty}
\frac{\lambda_{2x}(n)e^{-n/X}f_Y(n)}{n^{\sigma}}
\\
\ll_{\delta_1} &
\sum_{\substack{1\leq n \leq X\\ \text{non-sq}}}
\frac{n^{\delta_1/2}e^{-n/X}|f_Y(n)|}{n^{1/2+\delta_1}}
+
\sum_{\substack{n>X\\ \text{non-sq}}}
\frac{n^{\delta_1/2}e^{-n/X}|f_Y(n)|}{n^{1/2+\delta_1}}
\\
\leq &
\sum_{\substack{n\leq X \\ \text{non-sq}}}
\frac{|f_Y(n)|}{n^{(1+\delta_1)/2}}
+
\sum_{\substack{n\geq X \\ \text{non-sq}}}
\frac{X^{1/2}|f_Y(n)|}{n^{1+\delta_1/2}}
\\
\leq &
\Big(
\sum_{\substack{n\leq X \\ \text{non-sq}}}
|f_Y(n)|^2\Big)^{1/2}
\Big(
\sum_{\substack{n\leq X \\ \text{non-sq}}}
\frac{1}{n^{1+\delta_1}}\Big)^{1/2}
\\
&+
X^{1/2}
\Big(
\sum_{\substack{n> X \\ \text{non-sq}}}
\frac{|f_Y(n)|^2}{n^{1+\delta_1/2}}
\Big)^{1/2}
\Big(
\sum_{\substack{n> X \\ \text{non-sq}}}
\Big)^{1/2}
\\
\ll_{\delta_1}&
Y^{1/2}X^{1/2}(\log X)^2
\ll_{\delta_1}
Y^{(1+\delta_1)/2}(\log Y)^2
\ll
Y^{7/8}(\log Y)^2
\end{align*}
and we obtain
\begin{align}\label{TG}
&
\frac{1}{N(Y)}
\underset{|D|\leq Y}{{\sum}^{\dag}}
\psi_x(\mathcal{L}(\sigma, \chi_D))
\nonumber\\
=&
\frac{1}{N(Y)}
\sum_{n=1}^{\infty}
\frac{\lambda_{2x}(n^2)e^{-n^2/X}f_Y(n^2)}{n^{2\sigma}}
+
O_{\delta_1, R}(Y^{-1/8}(\log Y)^2+Y^{-\delta_1^2/4})
\nonumber\\
=&
\sum_{n=1}^{\infty}
\frac{\lambda_{2x}(n^2)\prod_{p\mid n}p(p+1)^{-1}}{n^{2\sigma}}
e^{-n^2/X}
+
O_{\delta_1, R}(Y^{-1/8}(\log Y)^2+Y^{-\delta_1^2/4})
\nonumber\\
=&
\sum_{n=1}^{\infty}
\frac{\lambda_{2x}(n^2)\prod_{p\mid n}p(p+1)^{-1}}{n^{2\sigma}}
+
O_{\delta_1, R}\Big(
\sum_{n=1}^{\infty}
\frac{|\lambda_{2x}(n^2)|}{n^{1+2\delta_1}}
|e^{-n^2/X}-1|
\Big)
\nonumber\\
&+
O_{\delta_1, R}(Y^{-1/8}(\log Y)^2+Y^{-\delta_1^2/4}).
\end{align}
Since $|e^{-a}-1| \leq 2a^{\delta_1/2}/\delta_1$ for $0< a < 1$, we see that the first error term in the right-hand side of \eqref{TG} is estimated by
\begin{align*}
\sum_{n=1}^{\infty}
\frac{|\lambda_{2x}(n^2)|}{n^{1+2\delta_1}}
|e^{-n^2/X}-1|
\ll_{\delta_1}&
\sum_{n<\sqrt{X}}\frac{1}{n^{1+3\delta_1/2}}\bigg(\frac{n^2}{X}\bigg)^{\delta_1/2}
+
\sum_{n\geq\sqrt{X}}\frac{1}{n^{1+3\delta_1/2}}
\\
\ll_{\delta_1}&
X^{-\delta_1/2}=Y^{-\delta_1^2/2}.
\end{align*}
Therefore we obtain
\begin{align*}
&
\frac{1}{N(Y)}
\underset{|D|\leq Y}{{\sum}^{\dag}}
\psi_x(\mathcal{L}(\sigma, \chi_D))
\\
=&
\sum_{n=1}^{\infty}
\frac{\lambda_{2x}(n^2)\prod_{p\mid n}p(p+1)^{-1}}{n^{2\sigma}}
+
O_{\delta_1, R}(Y^{-1/8}(\log Y)^2+Y^{-\delta_1^2/4}).
\end{align*}
\par
Secondly, we consider the case $\sigma>5/4$. 
In this case we know $\sum^*=\sum^{\dag}$.
We have some $\varepsilon_1>0$ for which 
\begin{align*}
&
\frac{1}{N(Y)}
\underset{|D|\leq Y}{{\sum}^{\dag}} 
\psi_x(\mathcal{L}(\sigma, \chi_D))
=
\frac{1}{N(Y)}
\sum_{n=1}^{\infty}\frac{\lambda_{2x}(n)}{n^{\sigma}}f_Y(n)
\\
=&
\frac{1}{N(Y)}
\sum_{\substack{1\leq n\\ \text{square}}}
\frac{\lambda_{2x}(n)}{n^{\sigma}}f_Y(n)
+
\frac{1}{N(Y)}
\sum_{\substack{1\leq n\\ \text{non-sq}}}
\frac{\lambda_{2x}(n)}{n^{\sigma}}f_Y(n)
\\
=&
\sum_{n=1}^{\infty}
\frac{\lambda_{2x}(n)\prod_{p\mid n}p(p+1)^{-1}}{n^{2\sigma}}
+
O(Y^{-\varepsilon_1})
\end{align*}
holds.
This is obtained by similar calculation to the above (see \cite{mourtada-murty}).
\end{proof}
%
%
\section{The proof of Theorem~\ref{main1}}\label{section_proof_main1}
We obtain Theorem~\ref{main1} for the case $\mathcal{L}(\sigma, \chi_D)=L'(\sigma, \chi_D)/L(\sigma, \chi_D)$ by Proposition~\ref{key_prop} and the works of Mourtada and Murty \cite{mourtada-murty}. 
We will prove Theorem~\ref{main1} in the case $\mathcal{L}(\sigma, \chi_D)=\log L(\sigma, \chi_D)$.
Put
\[
\widetilde{\mathcal{Q}_{\sigma}}(x)
=
\sum_{n=1}^{\infty} \frac{\lambda_{2x}(n^2)\prod_{p\mid n}(1+\frac{1}{p})^{-1}}{n^{2\sigma}}.
\]
Proposition~\ref{key_prop} yields
\[
\lim_{Y\to\infty}\frac{1}{N(Y)}
\underset{|D|\leq Y}{{\sum}^{\ast}} 
\psi_x(\log L(\sigma, \chi_D))
=
\widetilde{\mathcal{Q}_{\sigma}}(x)
\]
uniformly on $|x|\leq R$ and $\sigma\geq\frac{1}{2}+\delta$.
Here we know
\[
\widetilde{\mathcal{Q}_{\sigma}}(x)
=
\prod_p\bigg(1+\frac{p}{p+1}\sum_{r=1}^{\infty}\frac{\lambda_{2x}(p^{2r})}{p^{2r\sigma}}\bigg)
=
\prod_p\bigg(1+\frac{p}{p+1}\sum_{r=1}^{\infty}\frac{H_{2r}(ix)}{p^{2r\sigma}}\bigg).
\]
Let
\[
\widetilde{\mathcal{Q}_{\sigma, p}}(x)=
1+\frac{p}{p+1}\sum_{r=1}^{\infty}\frac{H_{2r}(ix)}{p^{2r\sigma}}
\]
and we write
$\widetilde{\mathcal{Q}_{\sigma}}(x)=\prod_p \widetilde{\mathcal{Q}_{\sigma, p}}(x)$.
We know
\begin{align*}
\widetilde{\mathcal{Q}_{\sigma, p}}(x)
=&
1+\frac{p}{2(p+1)}
\bigg(\sum_{r=1}^{\infty}\frac{H_{r}(ix)}{(p^{\sigma})^r}
+\sum_{r=1}^{\infty}\frac{H_{r}(ix)}{(-p^{\sigma})^r}
\bigg)
\\
=&
1+\frac{p}{2(p+1)}
\bigg(\frac{1}{(1-p^{-\sigma})^{ix}}+\frac{1}{(1+p^{-\sigma})^{ix}}
-2\bigg)
\\
=&
\frac{1}{p+1}+\frac{p}{2(p+1)}
\bigg(\frac{1}{(1-p^{-\sigma})^{ix}}+\frac{1}{(1+p^{-\sigma})^{ix}}\bigg).
\end{align*}
Since
\begin{align*}
&
\frac{1}{(1-p^{-\sigma})^{ix}}+\frac{1}{(1+p^{-\sigma})^{ix}}
=
\exp(-ix\log(1-p^{-\sigma}))+\exp(-ix\log(1+p^{-\sigma}))
\\
=&
\cos(x\log(1-p^{-\sigma}))-i\sin(x\log(1-p^{-\sigma}))
\\
&+
\cos(x\log(1+p^{-\sigma}))-i\sin(x\log(1+p^{-\sigma}))
\\
=&
2\exp\bigg(\frac{-ix\log(1-p^{-2\sigma})}{2}\bigg)
\cos\frac{x\log(1-p^{-\sigma})/(1+p^{-\sigma})}{2},
\end{align*}
we have
\begin{align*}
|\widetilde{\mathcal{Q}_{\sigma, p}}(x)|
\leq &
\frac{1}{p+1}+\frac{p}{p+1}
\left|\cos\bigg(\frac{x}{2}\log\bigg(1-\frac{2}{p^{\sigma}+1}\bigg)\bigg)\right|.
\end{align*}
For any prime number $p$, we see $|\widetilde{\mathcal{Q}}_{\sigma, p}|\leq 1$. 
For $x\in\mathbb{R}$, if $p$ satisfies
\[
\frac{\pi}{2^{\sigma}6}<\left|\frac{x}{2}\log \bigg(1-\frac{2}{p^{\sigma}+1}\bigg)\right|<\frac{2\pi}{3},
\]
then we have
\begin{equation}\label{cos}
|\widetilde{\mathcal{Q}_{\sigma, p}}(x)|
\leq
\frac{1+c(\sigma)p}{p+1}
\leq \frac{1+2c(\sigma)}{3},
\end{equation}
where
\[
c(\sigma)=\cos\frac{\pi}{2^{\sigma}6}.
\]
Since
\begin{align*}
t<-\log(1-t)
=&
t+\frac{t^2}{2}+\frac{t^3}{3}+\ldots
\\
<&
t+t^2+t^3+\ldots =\frac{t}{1-t}
\end{align*}
for $0<t<1$, we know
\begin{align*}
\frac{1}{p^{\sigma}} 
<
\frac{2}{p^{\sigma}+1} 
<& 
-\log \bigg(1-\frac{2}{p^{\sigma}+1}\bigg)
\\
<& 
-\log \bigg(1-\frac{2}{p^{\sigma}}\bigg)
<\frac{2}{p^{\sigma}-2}
<\frac{2}{p^{\sigma}-p^{\sigma}/2}=\frac{4}{p^{\sigma}}
\end{align*}
for $p^{\sigma}>4$.
Since $\sigma>1/2$, the prime $p>16$ in
\[
S_{\sigma}(x):=\left\{p \;:\; \frac{3|x|}{\pi} < p^{\sigma}\leq \frac{2^{\sigma}3|x|}{\pi}\right\} 
\]
satisfies
\begin{align*}
\frac{\pi}{2^{\sigma}6}\leq \frac{|x|}{2p^{\sigma}}
< -\frac{|x|}{2}\log \bigg(1-\frac{2}{p^{\sigma}+1}\bigg)
<\frac{2|x|}{p^{\sigma}}
<\frac{2\pi}{3}.
\end{align*}
The well-known fact of the asymptotic formula of $\pi(2x)-2\pi(x)$ yields
\begin{align*}
|S_{\sigma}(x)|
=&
\pi\bigg(2\bigg(\frac{3|x|}{\pi}\bigg)^{1/\sigma}\bigg)
-
\pi\bigg(\bigg(\frac{3|x|}{\pi}\bigg)^{1/\sigma}\bigg)
\\
=&
\pi\bigg(\bigg(\frac{3|x|}{\pi}\bigg)^{1/\sigma}\bigg)
+O\left(\bigg(\frac{3|x|}{\pi}\bigg)^{1/\sigma}
\bigg(\log \bigg(\frac{3|x|}{\pi}\bigg)^{1/\sigma}\bigg)^{-2}\right)
\\
>&
\alpha |x|^{1/(2\sigma)},
\end{align*}
for large $|x|$, where $\alpha$ is a positive constant depending on $\sigma$.
Therefore we obtain
\begin{align*}
\widetilde{\mathcal{Q}_{\sigma}}(x)
\leq \bigg(\frac{1+2c(\sigma)}{3}\bigg)^{|S_{\sigma}(x)|}
=&\exp\bigg(-(|S_{\sigma}(x)|)\log\frac{3}{1+2c(\sigma)}\bigg)
\\
\ll &\exp(-C_{\sigma}|x|^{1/(2\sigma)}),
\end{align*}
where $C_{\sigma}=\alpha\log 3(1+2c(\sigma))^{-1}>0$.
Hence 
\[
\int_{\mathbb{R}}|x^k|\widetilde{\mathcal{Q}_{\sigma}}(x) dx
\]
converges for all $k\geq 0$, which implies that its Fourier inverse
\[
\mathcal{Q}_{\sigma}(u)
=\frac{1}{\sqrt{2\pi}}\int_{\mathbb{R}}\exp(-iux)\widetilde{\mathcal{Q}_{\sigma}}(x)dx
\]
exists.
Moreover we know
\[
\int_{\mathbb{R}}\mathcal{Q}_{\sigma}(u)du
=
\widetilde{\mathcal{Q}_{\sigma}}(0)
=1
\]
by the definition of $\lambda_{2\xi}(n)$, 
and
\begin{align*}
\overline{\mathcal{Q}_{\sigma}(u)}
=&
\frac{1}{\sqrt{2\pi}}\int_{\mathbb{R}}\exp(iu x)\widetilde{\mathcal{Q}_{\sigma}}(-x)dx
\\
=&
\frac{1}{\sqrt{2\pi}}\int_{\mathbb{R}}\exp(-iu y)\widetilde{\mathcal{Q}_{\sigma}}(y)dy
=
\mathcal{Q}_{\sigma}(u).
\end{align*}
The same argument as the work of Mourtada and Murty~\cite{mourtada-murty} implies Theorem~\ref{main1}.
%
%

\bigskip
\noindent
Manami Hosoi:\\
Department of Mathematical and Physical Sciences,\\
Nara Women's University,\\
Kitauoya Nishimachi, Nara 630-8506, Japan.
\bigskip

\noindent
Yumiko Umegaki:\\
Department of Mathematical and Physical Sciences,\\
Nara Women's University,\\
Kitauoya Nishimachi, Nara 630-8506, Japan.\\
ichihara@cc.nara-wu.ac.jp
\end{document}